%% file: main.tex
\let\ORIlabel\label
\let\ORIrefstepcounter\refstepcounter
\AddToHook{package/hyperref/before}{%
   \let\label\ORIlabel
   \let\refstepcounter\ORIrefstepcounter
}
\documentclass[review,onefignum,onetabnum]{siamart171218}
\usepackage{graphicx}
\input{math_commands.tex}

\nolinenumbers
\usepackage{url}
\usepackage[numbers]{natbib}
\usepackage{amsmath, amsfonts}
\usepackage{amssymb}
\usepackage{comment}
\usepackage{booktabs}
\usepackage{siunitx}
\usepackage{subcaption}
\usepackage{algorithm}
\usepackage[noend]{algpseudocode}
\usepackage{tikz-cd}
\usepackage{hyperref}
\usepackage{cleveref}

\algnewcommand{\Assumption}{\item[\textbf{Assumption:}]}

\algnewcommand{\Problem}{\item[\textbf{Problem:}]}

\newcommand{\Mxn}{M_x^{(n)}}

\newcommand{\BR}{\mathbb{R}}

\newcommand{\norm}[1]{\left\lVert#1\right\rVert}
\DeclareMathOperator{\rank}{rank}
\DeclareMathOperator{\supp}{supp}
\DeclareMathOperator{\dist}{dist}

\newsiamremark{remark}{Remark}
\newsiamthm{assumption}{Assumption}
\newsiamremark{problem}{Problem}

\headers{On Moment-Based Recovery of Measures with Atomic and Continuous Parts }{Ruben Karapetyan, Shenyuan Ma, Ales Wodecki, Jakub Marecek}

\title{On Moment-Based Recovery of Measures with Atomic and Continuous Parts %
\thanks{Submitted to the editors \today.}}

\author{Ruben Karapetyan\thanks{Czech Technical University in Prague (\email{karaprub@fjfi.cvut.cz}).}
\and Shenyuan Ma\thanks{Czech Technical University in Prague (\email{ma.shenyuan@fel.cvut.cz}).} \and Ales Wodecki\thanks{Czech Technical University in Prague (\email{wodecale@fel.cvut.cz}).} \and Jakub Marecek\thanks{Czech Technical University in Prague (\email{jakub.marecek@fel.cvut.cz}).}}

\begin{document}

\maketitle

\begin{abstract}
Recovering probability measures from moments is a central theme in statistics and optimization. In particular, we focus on the recovery of measures from moments and pseudo-moments, which may come from solving the moment-SOS hierarchy in one dimension. A typical strategy when recovering a measure from moments is to verify the flat-extension property, which certifies that the underlying measure is finitely atomic and ultimately leads to recovery. For many classes of measures, however, the
flat extension never occurs and thus if one aims to recover the measure corresponding to the moments, assumptions need to be made. We formulate a new kind of recovery problem, where one assumes that the measure has compact support and a fulfills a mild separation criterion. The key feature of this recovery problem formulation is that it covers not only finitely atomic measures, but also measures with continuous components. We study this new problem and describe three situations in which different guarantees can be proven.
These guarantees are developed by studying the spectral representation of the Gelfand–Naimark–Segal construction and its connection to orthogonal polynomials, which ultimately allows us to provide several additional insights, which apply to algorithms widely used for the recovery of atomic measures from moments. Furthermore, the statements proven lead to novel algorithms, which we benchmark, further confirming the theoretical findings.

\end{abstract}




\section{Introduction}

\paragraph{Motivation}
The moment-sum-of-squares (moment-SOS) hierarchy \cite{Lasserre2001GlobalOptimization,Parrilo2003SemidefiniteProgramming} provides a motivational example in which the recovery of a measure from moments plays a pivotal role. It constitutes a systematic approach to polynomial optimization problems by replacing them with a sequence of semidefinite programs of increasing size.
At each relaxation order $d$, the hierarchy yields a matrix of pseudo-moments -- a positive semidefinite moment matrix $M_d(\mathbf{y})$.
Extracting the representing measure, or at least its support, from the truncated pseudo-moment data is therefore central to the practical use of the hierarchy \cite{HenrionLasserre2005DetectingGlobalOptimality,Nie2014_ATKMP,Gamertsfelder2025_GMPSoS}.

In the simplest case, the representing measure is finitely atomic, i.e., supported on finitely many points. There, the flat-extension theorem of Curto and Fialkow \cite{curto1996} provides a complete characterization: a condition on the ranks of consecutive moment matrices
$\rank M_d(\mathbf{y}) = \rank M_{d+1}(\mathbf{y})$
is both necessary and sufficient for the existence of a finitely-atomic representing measure, and its atoms can be extracted algorithmically via the Gelfand--Naimark--Segal (GNS) construction \cite{LopezQuijorna2021DetectingGNS,BurgdorfKlepPovh2016_OptimizationNoncommuting,KlepPovhVolcic2018_MinimizerExtractionIsRobust}. 
In the more challenging setting, where the measure has a continuous component (of non-zero Lebesgue measure), the strategy above fails and a fundamentally different approach is required to extract information.

\paragraph{Overview of prior work}
The recovery of measures from moment data has been studied from several distinct perspectives.
For measures supported on $r$ points, the GNS construction yields the atoms as the spectrum of a finite-dimensional multiplication operator \cite{curto1996,LopezQuijorna2021DetectingGNS}.
When the measure admits a density expressible as a finite polynomial, orthogonal polynomial expansion methods recover it exactly \cite{gautschi2004,soize2004}.
For measures satisfying a finite-moment condition, maximum entropy methods apply to certain classes of absolutely continuous distributions \cite{jaynes1957,cover2006}.
Nonparametric density estimation approaches, including kernel density estimation \cite{wand1995} and Gaussian mixture approximations \cite{hall2003,lindsay1995}, handle more general continuous measures, but require access to samples and provide only asymptotic guarantees.
When the measure cannot be determined from its moments, sample-based approaches are necessary.
Variational methods such as total variation minimization can provably recover spike measures from noisy data \cite{DeCastro2012BLASSO,Duval2015SparseSpikes}, but require separation conditions between atoms that may fail in practice.
Spectral methods including Prony’s method \cite{49090,harshman1970foundations}, annihilating filters, and MUSIC/ESPRIT-type algorithms are sensitive to noise \cite{Katz2024AccuracyProny,Li2021StabilityMUSICESPRIT}.
The finite rate of innovation framework \cite{vetterli2002sampling} assumes a parametric signal model whose misspecification degrades recovery.
Absent from this landscape is a framework and method that is able to work with the truncated pseudomoment sequence -- without sample access and without assuming 
the measure is finitely atomic.
Such a method should be able to handle a broader class of measures, namely ones containing a continuous and finitely atomic part.

\paragraph{Contributions}
We study the recovery of Borel probability measures on $\mathbb{R}$ of the form
$\mu = \mu_{pp} + \mu_{ac}$,
where $\mu_{pp}$ is a pure point part and $\mu_{ac}$ is an absolutely continuous part with disjoint, compact support.
Our main contributions are as follows.

\begin{enumerate}
\item \emph{Problem definition.} We present a problem formulation that allows the approximate recovery of the support of a measure with continuous and discrete components without the apriori knowledge of the structure ($\mu_{pp}$ or $\mu_{ac}$ might be zero) in Section \ref{sec:probdef}.
\item \emph{Novel results on the roots of orthogonal polynomials.} In Section \ref{sec:our_asymptotic_results} we provide novel results pertaining to the roots of orthogonal polynomials. These results provide new relations between the roots of orthogonal polynomials and the support of the measure under mild conditions. For ease of understanding, these results are summarized informally in Section \ref{sec:our_results_OGPoly}.
\item \emph{Novel algorithm.} We leverage the results on orthogonal polynomials to propose  Algorithm~\ref{alg:atom_extraction_theoretic}, which solves the problem proposed in Section \ref{sec:probdef}. This effectively estimates the atoms of
$\mu_{pp}$
and one or more intervals supporting the continuous part $\mu_{ac}$ of the measure.
\item \emph{Guarantees on algorithm.} 
Utilizing the results of  Section \ref{sec:our_results_OGPoly}, we show that the proposed algorithm  is polynomial-time when only a single continuous component is present and finite time where there are multiple continuous components under the BSS computational model.

\item \emph{Numerical illustration.} 
The proposed  algorithm is tested in Section~\ref{sec:num}, providing a very clear demonstration of many of the theoretical claims.

\item \emph{Operator-theoretic framework and connection to other algorithms}.
In Section \ref{sec:gen-curto-fialkow}, we show that the finite dimensional GNS multiplication operator, which is computed by a several existing solution-extraction algorithms, is unitarily equivalent to the finite-dimensional compression of the multiplication operator $M_x$  (on $L^2(\mathbb{R},\mu)$) on the subspace of polynomials of degree at most $n$, whose eigenvalues may be computed by finding the roots of the corresponding orthogonal polynomials. This shows that the theory developed in the present article can be applied to explain these solution-extraction algorithms for a larger class of measures than just finitely atomic ones.
Additionally, this places the GNS procedure squarely within the spectral theory of self-adjoint operators and provides a common framework for the finitely atomic and the mixed cases.

\end{enumerate}

\section{Preliminaries}
\label{sec:prelims}

For the readers convenience, we gather some details on the moment and truncated moment problems, which are most closely related to our problem.
Let $\mu$ be  a Borel probability measure on $E \subset\mathbb{R}$.
The measure generates a sequence
\begin{align*}
\mathbf{y} = (y_\alpha)_{\alpha\ge 0},
\qquad
y_\alpha := \int_{\mathbb{R}} x^\alpha \, d\mu(x),
\end{align*}
if all integrals $y_{\alpha}$ exist and are finite and we call $\mathbf{y}$ the moment sequence corresponding to $\mu$.

\emph{The moment problem} addresses the inverse problem.
Given a sequence $\mathbf{y}$, we ask whether there exists a measure $\mu$ that generates this sequence, whether such a measure is unique, and if so, whether it can be explicitly reconstructed.
This problem is classical and well studied; we refer to standard texts \cite{akhiezer1965classical} and \cite{Schmudgen2017} for a comprehensive treatment. Out of the many variants of the moment problem, the Hamburger moment problem is the one most closely related to our setting. In this problem, one studies the well definition and reconstruction of a measure from moments with support in $\mathbb{R}$. In particular, if one assumes that the measure associated with the moment sequence $\mathbf{y}$ is compactly supported, Carleman's condition  \cite{Schmudgen2017} is satisfied, and there exists an underlying measure uniquely defined by the moment sequence $\mathbf{y}$. 

In contrast to the classical setting described above, the computational literature focuses mainly on the truncated moment problem, where finitely many moments are available and the degree of the moment matrix is a fixed part of the problem \cite{Schmudgen2017, CurtoFialkow1998, Kimsey2021, curto1996}. Next, one imposes fairly strong structural assumptions on the measure itself to ensure the uniqueness of the reconstruction. Under these assumptions, the algorithms enable recovery of the full underlying measure, which can then be used. An example of such a use is the extraction of minimizers from moment-SOS relaxations \cite{HenrionLasserre2005DetectingGlobalOptimality, Gamertsfelder2025_GMPSoS, Nie2014_ATKMP} of polynomial optimization. (Note that without any additional assumptions, the problem \cite{Gitta2024} is not computable in models of computation more powerful than the Turing model.) 
In particular, one defines the moment matrix as follows
\begin{align*}
M_n(\mathbf{y})
:=
\left(
\int_{\mathbb{R}} x^i x^j \, d\mu(x)
\right)_{i,j=0}^{n}
\;\in\;
\mathbb{R}^{(n+1)\times(n+1)},
\end{align*}
where $n \in \mathbb{N}$ is called the order of the moment matrix $M_n(\mathbf{y})$. Curto and Fialkow \cite{CurtoFialkow1998} then show: 
\begin{theorem}[Flat extension]
\label{thm:flat-extension} 
A measure $\mu$ is finitely atomic if and only if the moment matrix corresponding to $\mu$ satisfies the flatness condition
\begin{equation}\label{eq:flat-extension-prop}
\rank M_n(\mathbf{y})
=
\rank M_{n+1}(\mathbf{y})
= r .
\end{equation}
for some $d, r\in \mathbb{N}$.
\end{theorem}
This theorem allows us to determine conclusively, whether the underlying measure is finitely atomic, which (assuming the model of computation of Blum, Schub and Smale, \cite{blum1989theory}) permits exact recovery of the measure from finite moment data using the GNS construction \citep{LopezQuijorna2021DetectingGNS}. From the onset, it is not clear how to even formulate such a recovery problem when the measure is not finitely atomic.

\section{Problem definition}
\label{sec:probdef}

It is no surprise that in practice we sometimes encounter measures, which are not finitely atomic but still need to be recovered. These measures might not be purely continuous either, which prevents the use of many of the well crafted methods that deal with compact support continuous measures mentioned in the introduction. We define a well posed problem that may be solved in finite time with guarantees which leads to the recovery of such measures. Before we define the problem, we succinctly formulate the assumptions on the measure.

\begin{definition}[Support of a measure]
Let $\mu$ be a Borel measure
on $\R$, then its support is defined as
\begin{align*}
    \supp \mu = \{ x \in \R : \mu(U) >0 \text{ for every open } U \ni x \}.
\end{align*}
\end{definition}

\begin{assumption}\label{as:measure_base}
Let $\mu$ be the Borel probability measure on $\R$ one wishes to recover.
The measure $\mu$ has a canonical decomposition by the Lebesgue decomposition theorem which reads $\mu = \mu_{pp} + \mu_{sc} + \mu_{ac}$, where $\mu_{pp}$, $\mu_{sc}$ and $\mu_{ac}$ are the pure point (or atomic), singularly continuous and absolutely continuous parts of the measure, respectively. We assume $\mu$ has the following properties:
\begin{itemize}
    \item $\mu$ is supported on a compact set classifying the recovery problem as a Hausdorff moment problem. The bound on the support of $\mu$ will be labeled $B > 0$.
    \item $\mu_{sc} = 0$, which implies that
    \begin{equation}\label{eq:measure-form}
        \mu = \mu_{pp} + \mu_{ac},
    \end{equation}

\item 
the support of the continuous part consists of $m$ disjoint intervals
\begin{align}
\label{eq:connected_component_interval}
    \supp \mu_{ac} = [a_1,b_1] \cup [a_2,b_2] \cup ... \cup [a_m,b_m],
\end{align}

\item 
the support of the pure point (atomic) part consists of $r \in \mathbb{N}_0$ points
\begin{align*}
\label{eq:connected_component_atom}
    \supp \mu_{pp} = \{x_1,x_2,...,x_r\},
\end{align*}

\item the supports of $\mu_{pp}$ and $\mu_{ac}$ are disjoint,
\item 
the density satisfies
\begin{align*}
    d \mu_{ac}(x) \geq c_0 >0, \quad x \in \supp \mu_{ac}.
\end{align*}
\end{itemize}
\end{assumption}

The following, additional assumptions on the measure $\mu$ significantly simplify our approach to the moment problem as we will see later.

\begin{assumption}
\label{as:atoms_outside}
    Let $\mu$ satisfy Assumption \ref{as:measure_base}.
    We additionally assume that the $r$ atoms lie outside the convex hull of the set supporting the continuous part, i.e.
    \begin{align*}
        \supp \mu_{pp} \cap \mathrm{conv hull}(\supp \mu_{ac}) = \emptyset.
    \end{align*}
\end{assumption}

\begin{assumption}
\label{as:single_interval}
    Let $\mu$ satisfy Assumption \ref{as:measure_base}.
    We additionally assume that the continuous part $\mu_{ac}$ is supported on a single interval,
    \begin{align*}
        \supp \mu_{ac} = [a,b], \quad -\infty < a < b < +\infty.
    \end{align*}
\end{assumption}

\begin{remark}
Since the measure has compact support, Carleman's condition is satisfied apriori. This ensures that the moment sequence generated by $\mu$ which satisfies \ref{as:measure_base} is associated with a unique measure.
\end{remark}

\begin{definition}\label{def:paddedSpacing}
Let $\mu$ be a measure satisfying Assumption \ref{as:measure_base}.
The characteristic separation distance $\Delta>0$ of $\mu$ is the minimum of mutual distances between all the connected components (\ref{eq:connected_component_interval}), (\ref{eq:connected_component_atom}) of $\supp \mu$.
\end{definition}
Notice in particular, that $\Delta$ is always well defined for measures satisfying Assumption \ref{as:measure_base}. The problem that we propose is the following. 
\begin{problem}\label{prob:reconstr}
Let a measure $\mu$ be such that:
\begin{itemize}
    \item Assumption \ref{as:measure_base} holds with support bound $B > 0$,
    \item $\mu$ has a characteristic separation distance $\Delta > 0$.
\end{itemize}
Let $\varepsilon > 0$, $2\varepsilon<\Delta$, be desired maximal error of reconstruction and let $d_H$ denote Hausdorff distance. 
We say that the triple $(A, A_{pp}, A_{ac})$ solves the reconstruction problem up to  $\varepsilon$ precision if
\begin{itemize}
    \item $A_{pp}, A_{ac}$ were obtained using only finite, although arbitrarily large amount of moments.
    \item $A=A_{pp}\cup A_{ac}$,
    \item (precision condition) $d_H(\supp \mu,A) < \varepsilon$, where $d_H$ is Hausdorff distance,
    \item (localization condition) for any atom $a$ of the measure $\mu$, $A_{pp}\cap B(a,\varepsilon)$ has exactly one element and $A_{ac}\subset \supp \mu_{ac} + B(0, \varepsilon)$.
\end{itemize}

\begin{remark}
    Let us note that in Problem \ref{prob:reconstr}, we assume only the knowledge of the moments and whether Assumptions \ref{as:measure_base}, \ref{as:atoms_outside} or \ref{as:single_interval} hold.
    Generally, we do not need to know the number $m$ of continuous parts and the number of roots $r$ in advance.
    Although, practically, having this information in advance might simplify the problem.
\end{remark}

\end{problem}
The reconstruction problem, as defined above has several key features. First of all it allows us to define conditions under which we can reconstruct the measure without assuming anything about the structure of the support (number of continuous parts or atoms). Second, the solution of the problem is not only a discrete set of points that approximate the support, but also information that allows us to localize the discrete and continuous parts of the support.


\section{Our results}
\label{sec:results}
All of our efforts culminate in a polynomial-time algorithm for solving
Problem~\ref{prob:reconstr}. Assume that we are given a sequence of moment
matrices \(\{M_n(\mathbf{y})\}_n\) generated by a measure \(\mu\) satisfying
Assumption~\ref{as:measure_base}, with \(\mu_{ac}\neq 0\). Since such a
measure is rather general and not finitely atomic, we adopt an
asymptotic approach.

We consider three levels of assumptions on \(\mu\), ranging from the most
general to the most restrictive (Assumptions \ref{as:measure_base} - \ref{as:single_interval}). The more general the measure is, the more
difficult the extraction problem becomes. 
In each case, we apply the same basic algorithmic framework, augmented by additional steps needed to handle the extra generality of the measure.

In Section~\ref{sec:our_results_OGPoly}, we introduce our results on orthogonal
polynomials. In Section~\ref{sec:deriving_algorithm}, we explain how these theoretical results lead
to a construction that solves Problem~\ref{prob:reconstr} under the respective
assumptions on \(\mu\). Finally, in Section~\ref{Sec:Algorithms}, we present the resulting
algorithm and prove that, under the stated assumptions, its runtime is polynomial.

\subsection{Our results on orthogonal polynomials}
\label{sec:our_results_OGPoly}
\begin{definition}[Monic orthogonal polynomials]
\label{def:OGPoly}
    Suppose $\mu$ satisfies Assumption \ref{as:measure_base}.
    Then, we define the set of monic  polynomials $\{P_n\}_{n\geq0}$ orthogonal with respect to the measure $\mu$, which is uniquely given by the conditions:
    \begin{align*}
        \deg P_n= n, \\
        \text{the coefficient of } x^n \text{ in } P_n \text{ is } 1, \\
        \int P_n(x)P_m(x) d\mu(x) = K_n \delta_{m,n}.
    \end{align*}
    \end{definition}
\begin{definition}[Root sequence] \label{def:root_sequence}
    Suppose $\mu$ is a measure $\mu$ as specified above.
    Define the set of roots of the orthogonal polynomials with respect to  $\mu$ as $\mathcal{R}_{n}=\left\{ x\in\mathbb{R}:P_{n}\left(x\right)=0\right\}$.
Then, we call any sequence $(x_{n})_{n\geq n_{0}}$, where $n_0 \in \mathbb{N}$ such that $x_{n}\in \mathcal{R}_{n}$ for all $n \geq n_0$ a root sequence.
\end{definition}

To analyze the structure of the measure $\mu$ and extract information about its support, it is natural to consider the associated monic orthogonal polynomials. 
Assuming that $\mu$ has infinite support, the orthogonalization procedure does not terminate, and thus a full sequence of orthogonal polynomials is available.
We propose to  analyze the zeros of these polynomials (i.e. root sequences \ref{def:root_sequence}).
We show that the zeros exhibit convenient asymptotic behavior which can reveal the atoms of $\mu_{pp}$ and the support of the continuous part $\supp \mu_{ac}$.
The following result is an informal version of Theorems \ref{lemma:atom_convergence}, \ref{thm:distinguish}, \ref{thm:distinguish_outside}, \ref{thm:distinguish_general}, which encapsulate the ideas of how we treat the atomic and the continuous part, respectively.

\begin{theorem}[Asymptotic behavior of roots, informal] 
\label{thm:informal_spectral_inclusiveness}

Suppose $\mu$  is a measure satisfying Assumption \ref{as:measure_base}, 
\begin{align*}
    \supp \mu_{ac} = \cup_{j=1}^m  [a_j,b_j] \not= \emptyset, \quad \supp \mu_{pp} = \{x_1,x_2,...,x_r\}, \quad r \in \mathbb{N}_0.
\end{align*}
Then the following statements hold:
\begin{itemize}
    \item The flatness condition (\ref{eq:flat-extension-prop}) is never satisfied.
    \item Given $\varepsilon>0$, there exists a sufficiently large $n_0$ such that for arbitrary $n\geq n_0$, the roots of $P_n$ have the following properties:
    \begin{itemize}
        \item We observe precisely $m$ bulks of at least three $\varepsilon$-close roots $\varepsilon$-accurately approximating  $\supp \mu_{ac}$.
        \item We observe anywhere from $r$ up to $r+m-1$ $\varepsilon$-isolated roots.
        \item From the isolated roots, we can identify the $r$ atoms $\varepsilon$-accurately approximating $\supp \mu_{pp}$ by inspecting $P_{n+1}$.
    \end{itemize}
\end{itemize}
If $\mu$ additionally satisfies Assumption \ref{as:atoms_outside}:
\begin{itemize}
    \item If $n \geq n_0$, we can decide whether an isolated root is a polluting root by mere inspection of $P_n$ (there is no need to analyze $P_{n+1}$).
\end{itemize}
Moreover, if $\mu$ in addition satisfies Assumption \ref{as:single_interval}:
\begin{itemize}
    \item 
    Given $\varepsilon>0$, for $n \geq n_0 = O(1/\varepsilon)$, within the roots of $P_n$, we have one bulk of $\varepsilon$-close roots, which $\varepsilon$ accurately approximates $\supp \mu_{ac}$.
    Moreover, there are $r$   $\varepsilon$-isolated root and they   $\varepsilon$ accurately approximate their corresponding atoms.
    \item For large enough $n=O(1/\varepsilon)$, all roots of $P_n$ fall in $\supp \mu + B(0,\varepsilon)$.
    \item  Every atom has a unique root sequence associated with it.
    This root sequence converges to the atom exponentially.
\end{itemize}
\end{theorem}

\subsection{Arriving at an algorithm}
\label{sec:deriving_algorithm}
Our aim is to develop an algorithm that solves Problem \ref{prob:reconstr}.
We expect that the structure of the underlying measure $\mu$ dramatically affects the difficulty of the problem.
Theorem \ref{thm:informal_spectral_inclusiveness} confirms this expectation - in the case of Assumption \ref{as:single_interval}, when we only have one continuous component in $\mu_{ac}$, we have the strongest theoretical result.
Given $\varepsilon>0$, we have two types of roots of $P_n$ if $n\geq n_0 = O(1/\varepsilon)$:
\begin{itemize}
    \item one bulk of $\varepsilon$-close roots $(y_j)_{j=1}^{n-r}$,
    \item precisely $r$ $\varepsilon$-isolated roots $(z_i)_{i=1}^r$.
\end{itemize}  
These two parts approximate $\supp \mu$ $\varepsilon$-accurately.
Then, we can solve Problem \ref{prob:reconstr} by setting \begin{align*}
    A_{ac} &= [ \min_j(y_j)_{j=1}^{n-r}, \max_j(y_j)_{j=1}^{n-r}] , 
    \\
    A_{pp} &= (z_i)_{i=1}^r.
\end{align*} 

If we weaken the assumptions on $\mu$ to satisfy Assumption \ref{as:atoms_outside} (atoms lie outside convex hull of $\supp \mu_{ac}$), it still holds that we can solve Problem \ref{prob:reconstr} by computing roots of $P_n$ for some large $n$.
However, in this case, we do not have a provable bound for $n$ in terms of $\varepsilon$.
Nevertheless, we observe $m$ bulks or clusters of  roots that are $\varepsilon>0$ close.
These must be associated with the continuous part.
Moreover, an $\varepsilon$-isolated root either lies outside of $\mathrm{convhull } \ \supp \mu_{ac}$ and we are guaranteed it is an $\varepsilon$-accurate approximation of an atom or the isolated root lies inside $\mathrm{convhull } \ \supp \mu_{ac}$ and is necessary a pollution since we assume no atoms inside $\mathrm{convhull } \ \supp \mu_{ac}$.
We can then solve Problem \ref{prob:reconstr} in a similar manner as above. 
Computing the minimum and maximum of the $m$ bulks will give us a valid approximation of $\supp\mu_{ac}$ and the $r$ roots outside these bulks will be a valid approximation of $\supp \mu_{pp}$.

If we further weaken the assumptions to $\mu$ satisfying only Assumption \ref{as:measure_base},
we still observe $m$ bulks of $\varepsilon$-close roots but the ruling out of pollution is more delicate.
We have at most $r+m-1$ candidates for $r$ atoms and in order to be able to distinguish the pollution from the atomic roots, we need to inspect roots of two consecutive orthogonal polynomials.

\subsection{Algorithms leveraging the theory}
\label{Sec:Algorithms}

Let $\mu$ satisfy Assumption \ref{as:measure_base}.
Suppose we have a sequence of moments $\{M_n(\mathbf{y}))\}_{n}$ generated by $\mu$.
Our goal now is to locate the atomic part of the measure $\mu_{pp}$ and the support of the continuous part $\mu_{ac}$.

Algorithm \ref{alg:atom_extraction_theoretic}  produces the desired approximation of $\supp \mu$.
The key steps are  building the monic polynomials orthogonal \ref{def:OGPoly} with respect to the measure $\mu$, computing their zeros and then analyzing them.

First, we describe the algorithm $\texttt{roots}$ that ensures the computation of the zeros of orthogonal polynomials from the moment sequence. The procedure is known in the literature as discretized Stieltjes procedure \cite{gautschi_orthogonal_1985}, which is based on the observation that the three term recurrence relation (see equation (\ref{eq:three_term}) in Section \ref{Sec:orthogonal_polynomials}) can be computed directly if the integration oracle of a measure $\mu$ is available (as explained below). 
\begin{algorithm}
\caption{ $\texttt{roots}$: Computation of roots of orthogonal polynomials from an integration oracle}
\label{alg:roots_from_moments}
\begin{algorithmic}[1]
\Require A maximum degree $n$, an integration oracle accurate up to degree $2n-1$.
\Ensure List $ \mathcal{R}$ of roots of the $n$-th monic orthogonal polynomial $P_n$. (\ref{def:OGPoly}).
\State Initialize polynomial $P=1,Q=0$.
\For{$j = 0$ to $n-1$}
    \State $\zeta_j = \langle P,P\rangle_\mu$, computed with oracle.
    \State $\alpha_j = \frac{\langle xP,P\rangle_\mu}{\zeta_j}$, computed with oracle.
    \If{$j\geq 1$}
        \State $\beta_j = \frac{\zeta_j}{\zeta_{j-1}}$.
    \EndIf
    \State Update polynomials $(P,Q)\leftarrow ((x-\alpha_j)P-\beta_j Q,P)$.
    \State Build the tridiagonal Jacobi matrix $J_n$ (see Definition \ref{def:Jacobi_matrix}).
\EndFor
\State compute and \textbf{return} the list of eigenvalues of $J_n$, i.e. the roots of polynomial $P_n$.
\end{algorithmic}
\end{algorithm}

Let $\langle p,q\rangle_{\mu}=\int p(x)q(x)\mathrm{d}\mu$ denote the inner product induced by the measure $\mu$, we compute the three term recurrence coefficients by
\begin{align}
\label{eq:recurrence_alpha}
    \alpha_i&=\frac{\langle xP_i, P_i\rangle_{\mu}}{\langle P_i, P_i\rangle_{\mu}},\forall i=0,\dots,n-1,\\
\label{eq:recurrence_beta}
    \beta_i&=\frac{\langle P_i,P_i\rangle_{\mu}}{\langle P_{i-1},P_{i-1}\rangle_{\mu}},\forall i=1,\dots,n-1,
\end{align}
where the integration of polynomials is allowed by the moment matrix $M_n$ as
\begin{align*}
    \langle p,q\rangle_{\mu} = \vec{p}^T M_n \vec{q}, \quad \forall p,q \in \R[x]_n.
\end{align*}
Since by convention $P_{-1}=0$, the value of $\beta_0$ is unimportant.
Once the sequences $\{\alpha_i\}_{i=0,\dots,n-1}$ and $\{\beta_i\}_{i=1,\dots,n-1}$ are built, they determine the tridiagonal Jacobi matrix $J_n$ corresponding to $\mu$ (see Definition (\ref{def:Jacobi_matrix})).
The eigenvalues of this matrix $J_n$ are exactly the roots of the polynomial $P_n$ (see Theorem \ref{thm:spectrum+zeros}).

Below, we present the algorithm 
$\texttt{SupLoc}$, which is precisely the aforementioned step of analyzing the computed zeros.
The guiding idea is based on Theorem \ref{thm:informal_spectral_inclusiveness} (which is an informal version of Theorems \ref{thm:distinguish}, \ref{thm:distinguish_outside}, \ref{thm:distinguish_general}) and explained in Section \ref{sec:deriving_algorithm}.
The theoretical guarantees for Algorithm \ref{alg:atom_extraction_theoretic} are encapsulated in Theorem \ref{thm:algorithm_proof}.

\begin{algorithm}
\caption{$\texttt{SupLoc}$: Estimate atoms and the interval supporting the continuous part of the measure from its moments }
\label{alg:atom_extraction_theoretic}
\begin{algorithmic}[1]

\Require Desired precision $\varepsilon>0$.
\Require Moment matrix $ M_{N}$, $N$ depends on the assumptions.

\Assumption
Let the moments come from a measure $\mu$ satisfying Assumption \ref{as:measure_base}.
\Assumption \textbf{Additionally,} the measure $\mu$ also satisfies Assumption \ref{as:single_interval} or Assumption \ref{as:atoms_outside}.

\Ensure The location of the discrete part $\mu_{pp}$ and inner approximation of the interval supporting $\mu_{ac}$.

\For{$n=1,2,...,N$}
    \If{Flatness condition $\rank M_n(\mathbf{y})
=
\rank M_{n+1}(\mathbf{y})$ holds}
        \State 
        Necessarily, $\supp \mu_{ac} = \emptyset$.
        \State Perform the truncated GNS construction to obtain the atoms $\supp \mu_{pp}$.
        \State
        \Return $\supp \mu_{pp}$
    \EndIf
\EndFor

\State
Initialize containers:
\begin{align*}
\mathcal A_N \gets \varnothing, \qquad 
\mathcal I_N \gets \varnothing \qquad \text{(corresponding to atoms and intervals, respectively)}.
\end{align*}

\State Use algorithm \ref{alg:roots_from_moments} $\texttt{roots}$  to obtain the zeros $\mathcal{R}_{N}$ of the orthogonal polynomial $P_{N}$ from the moment matrix $M_{N}$.

\For{$r \in \mathcal{R}_{N}$}
    \If{$\exists \tilde{r} \in \mathcal{R}_N : |r - \tilde{r}| < \varepsilon$}
        \State Append $r$ to $\mathcal{I}_N$
    \Else
        \State Append $r$ to $\mathcal{A}_N$
    \EndIf
\EndFor

\State
Sort the roots in $\mathcal{I}_N$ to be $y_1 <y_2 <...<y_{m-1}<y_m $.

\State
Set the number of continuous parts $k\gets1$
and
$a_{1} \gets y_1$.

\If{Assumption \ref{as:single_interval} holds}
\State
Set  $\ b_1 \gets y_m$
\State
\Return the atoms $\mathcal{A}_N$ with the interval $[a_1,b_1]$

\Else{ (Assumption \ref{as:single_interval} does not hold)}

\For{$j \in \{2,3,...,m-1\}$}
    \If{$|y_{j+1} - y_{j}| \geq \varepsilon$}
    \State 
    $b_{N_k} \gets y_j$.
    
        \State $k \gets k+1$.
        \State
        $a_{N_{k}} \gets y_{j+1}$.
    \EndIf
\EndFor
\State 
$b_{N_k} \gets y_m$.

\If
{Assumption \ref{as:atoms_outside} holds}

\State \Return the atoms $\mathcal{A}_N$ together with the $k$ intervals $[a_{N_1},b_{N_1}], [a_{N_2},b_{N_2}],...,[a_{N_k},b_{N_k}]$ approximating the $\supp \mu_{ac}$.

\Else{ (none of Assumptions \ref{as:atoms_outside}, \ref{as:single_interval} hold)}
\State
Run steps 7-24 for $ M_{N+1}$
to obtain $\mathcal{A}_{N+1}$.

\State
Use the $\rho = \varepsilon^2/(\varepsilon+\sqrt2 a_\infty)$ separation criterion from Theorem~\ref{thm:distinguish_general} to filter out the pollution.
\For{$r \in \mathcal{A}_{N}$}
    \If{$\exists \tilde{r} \in \mathcal{A}_{N+1} : |r - \tilde{r}| < \rho$}
        \State Keep $r$ in $ \mathcal{A}_N$
    \Else
        \State Discard $r$.
    \EndIf
\EndFor

\State
\Return $\ \mathcal{A}_N, [a_{N_1},b_{N_1}], [a_{N_2},b_{N_2}],...,[a_{N_k},b_{N_k}]$ approximating the $\supp \mu$.

\EndIf
\EndIf
\end{algorithmic}
\end{algorithm}

\begin{theorem} \label{thm:algorithm_proof}
    Suppose we have a measure $\mu$ satisfying Assumption \ref{as:measure_base} with the characteristic separation distance $\Delta >0$.
    Let us assume the BSS with a square root model of computation.
    Then, for any $\varepsilon>0 $, $0<2\varepsilon < \Delta $, 
    we can solve Problem \ref{prob:reconstr} using Algorithm \ref{alg:atom_extraction_theoretic}: 
    \begin{enumerate}
        \item \label{item:1interval} in a polynomial amount of operations $O(1/\varepsilon^3)$ from $M_N(\mathbf{y})$, where 
    $N= O( 1 / \varepsilon )$
    if $\mu $ additionaly satisfies Assumption \ref{as:single_interval}.
        \item  \label{item:atoms_outside}
        in a finite amount of  operations from $M_N(\mathbf{y})$, where 
    $N$ is large enough
    if $\mu $ additionaly satisfies Assumption \ref{as:atoms_outside}.
    
        \item \label{item:last}
         in a finite amount of  operations  from $M_N(\mathbf{y})$, $M_{N+1}(\mathbf{y})$ where 
    $N$ is large enough if neither Assumption
    \ref{as:atoms_outside} nor \ref{as:single_interval} are satisfied.
    \end{enumerate}
\end{theorem}
\begin{proof}  
 We analyze the complexity in the unit-cost algebraic model over $\mathbb{R}$ (the Blum--Shub--Smale model).
    In this model, each arithmetic operation has unit cost.
    Let us note that when the measure is finitely atomic and the flat extension (\ref{eq:flat-extension-prop}) holds, the GNS gives us an exact location of the atoms with a number of operations that is cubic in the number of atoms.
    In what follows, we analyze the $\mu_{ac}\not = 0$ case, in which  (\ref{eq:flat-extension-prop}) never holds.

    The general approach is the same. Generate orthogonal polynomials from the moment matrices and compute their roots.
    Given a moment matrix $M_N$, we provide an overview of the main algebraic operations involved and their computational cost (see \cite{Demmel1997} for details).
    \begin{enumerate}
        \item Cholesky factorization of a $(N+1)\times(N+1)$ matrix: $O(N^3)$.
        \item Inversion of a triangular $(N+1)\times(N+1)$ matrix: $O(N^3)$.
        \item Computing $N$ recurrence coefficients 
        $\alpha_i, \beta_i, i =0,1,...,N-1$
        \ref{eq:recurrence_alpha}, 
        \ref{eq:recurrence_beta}: $O(N^2)$
        \item 
        Solve the eigenvalue problem for a $(N+1)\times(N+1)$ matrix $J_N$ \ref{eq:Jacobi_matrix}: $O(N^3)$.
        \item 
        Sort and compute the distance of neighboring roots in $\mathcal{R}_N$: $O(N \log N)$.
    \end{enumerate}
    Together, this yields $O(N^3)$ operations.
    Since $N= O(1/\varepsilon),$
    the total complexity is $O(1/\varepsilon^3)$.

    Ad \ref{item:1interval}:
    Due to Theorem \ref{thm:distinguish}, we know how large $N$ we need in order to identify the $\supp \mu_{ac}$ and $\supp \mu_{pp}$ from the roots of $P_N$ up to $\varepsilon$ precision.
    To satisfy the assumptions of Theorem \ref{thm:distinguish}, we put $\delta = \varepsilon$ and assume large enough $N= O(1/\varepsilon)$.
    We then proceed with the algorithm for such $M_N$.
    This yields $O(N^3)=O(1/\varepsilon^3)$ operations to solve Problem \ref{prob:reconstr}.

    Ad \ref{item:atoms_outside}:
    We use Theorem \ref{thm:distinguish_outside},
    which distinguishes two types of roots.
    First, those that have a $\delta$ close neighbor within the roots and the endpoints of these $\delta$ close bulks work as a valid $\delta$ accurate approximation of the corresponding interval.
    Or second, they do not have a $\delta$ close neighbor.
    Due to Assumption \ref{as:atoms_outside}, we can argue that 
    the roots without a neighbor
    either correspond to an atom or we can rule them out as pollution as we expect no atoms in between the continuous parts (and that is the only place where pollution can occur).  
    The complexity is still $O(N^3)$.
    However, now, we do not have a provable bound for $N$.
    We only know $N$ is finite.

    Ad \ref{item:last}: In this case, we use the criteria from Theorem \ref{thm:distinguish_general} to classify the roots of $P_N$ and to guarantee that the result solves Problem \ref{prob:reconstr} in $O((N+1)^3)=O(N^3)$ steps as well.
    We do not however have any provable bounds for the degree $N$ needed.
    We can only note that in this situation, the $N$ needed is significantly larger than if we added Assumption \ref{as:atoms_outside} or \ref{as:single_interval}.
\end{proof}

We benchmark this algorithm in Section \ref{sec:num}.

\section{New results on orthogonal polynomials}
\label{Sec:orthogonal_polynomials}

This section is devoted to studying the behavior of zeros of monic orthogonal polynomials (Definition \ref{def:OGPoly}) corresponding to a measure $\mu$ satisfying Assumption \ref{as:measure_base} and optionally Assumptions 
\ref{as:atoms_outside} or \ref{as:single_interval}. First, we introduce some standard results that play a key part in our derivations (Sections \ref{sec:three_term}, \ref{sec:standard_zeros_ogpoly}), then we provide the exact statement of our novel results in Section \ref{sec:our_asymptotic_results}.

\subsection{Three term recurrence}
\label{sec:three_term}
Let $\{P_n\}_{n\in \mathbb{N}}$ be the sequence of monic orthogonal polynomials \ref{def:OGPoly} of a measure $\mu$, then it satisfies the three term recurrence relation \cite{Szego}:
\begin{align} \label{eq:three_term}
    xP_{j}(x) =
    P_{j+1}(x) +
     \alpha_j P_j(x) + \beta_{j} P_{j-1}(x),\forall j
\end{align}
with the convention $P_{-1} =0$. Let us note that the coefficients $\alpha_j, \beta_j$ contain sufficient amount of information about the underlying measure $\mu$.

\begin{definition}
\label{def:Jacobi_matrix}
    Given $\mu$ and (\ref{eq:three_term}), we define the $n$-th tridiagonal Jacobi matrix associated with $\mu$ \cite{gautschi2004} as
\begin{align}
\label{eq:Jacobi_matrix}
J_n =
\begin{pmatrix}
\alpha_0 & \sqrt\beta_1 & 0 & \cdots & 0 \\
\sqrt\beta_1 & \alpha_1 & \sqrt\beta_2 & \ddots & \vdots \\
0 & \sqrt\beta_2 & \alpha_2 & \ddots & 0 \\
\vdots & \ddots & \ddots & \ddots & \sqrt\beta_{n-1} \\
0 & \cdots & 0 & \sqrt\beta_{n-1} & \alpha_{n-1}
\end{pmatrix}.
\end{align}
\end{definition}

\begin{theorem}[\cite{simon2005opuc1}] \label{thm:spectrum+zeros}
    Suppose $\mu$ is a Borel measure  on $\R$.
    Let $(P_n)_{n=0}^{\infty}$ be the sequence of the polynomials orthogonal with respect to $\mu$.
    Let $(J_n)_{n=1}^{\infty}$ be the sequence of tridiagonal Jacobi matrices associated with $\mu$.
    Then
    \begin{align*}
        \det (xI - J_n) = P_n(x), \quad n \in \mathbb{N},
    \end{align*}
    and therefore, the spectrum of $J_n$ is equal to the set of zeros of $P_n$.
\end{theorem}

\subsection{Zeros of orthogonal polynomials}
\label{sec:standard_zeros_ogpoly}
We follow with several standard results from the theory of orthogonal polynomials  \cite[Theorem 3.3.1., Theorem 3.3.2., Theorem 6.1.1.]{Szego} and \cite[Theorem 2.4.]{Freud1971}. These illustrate the practical properties of orthogonal polynomials and will be leveraged in our own theorems and algorithm.

\begin{theorem}  [Properties of OG polynomials] \label{thm:properties_og_poly}   
    Let $\mu$ be a positive Borel measure on $\mathbb{R}$ with infinite support and finite moments of all orders.
Let $\{P_n\}_{n\ge 0}$ be the corresponding orthonormal polynomials with respect to $\mu$.
 Let $[a',b']$ be a subinterval of $[a,b]$.
Then the orthogonal polynomials have the following properties:
\begin{enumerate}
    \item
    \label{thm:roots_and_support}
    \textbf{Roots and support of the measure:}  For every $n\ge 1$, the polynomial $P_n$ has exactly $n$ real, simple zeros, and all of them lie in the
closed convex hull of the support of $\mu$.
    \item
    \label{thm:density_of_zeros}
    \textbf{Density  of zeros:} If $\int_{a'}^{b'} d\mu(x)>0$ and $n$ is sufficiently large, every polynomial $P_n(x)$ has at least one zero in $[a',b']$.
    \item 
    \label{thm:interlacing}
    \textbf{Interlacing property:}
    Between two zeros of $P_n(x)$, there is at least one zero of $P_m(x)$, $\forall m>n$.
    \item 
    \label{thm:measure_0_sets}
    \textbf{Roots and measure 0 sets:}
    If $\int_{a'}^{b'} d\mu(x)=0$,  every polynomial $P_n(x)$ has at most one zero in $[a',b']$.
\end{enumerate} 
    
\end{theorem}

\subsection{Asymptotic behavior of zeros of orthogonal polynomials}
\label{sec:our_asymptotic_results}
This subsection is devoted to novel results on the zeros of orthogonal polynomials that lead to provably tractable algorithms for the extraction of the support of the measure (\ref{Sec:Algorithms}).
Namely, we show that for a measure $\mu$ satisfying Assumption \ref{as:measure_base}, for each atom,  
there is a unique root sequence \ref{def:root_sequence} that converges to the atom  exponentially quickly.
Moreover, we leverage results by \cite{Simon2007_equilibrium}, \cite{Totik2009} implying that for such $\mu$, the spacing of zeros on the interval supporting $\mu_{ac}$ is asymptotically linear, which results into bulks of roots that densely populate $\supp \mu_{ac}$.

For our first result, we need to  restrict ourselves to $\mu$ satisfying Assumption \ref{as:single_interval}, without loss of generality we can assume that $\supp\mu_{ac}=[-1,1]$. Since $\mu_{ac}$ is assumed to be absolutely continuous with respect to the Lebesgue measure, it admits a density function $\alpha\in C(-1,1)$ which is given by the Radon-Nikodym derivative. 
In order to prove the following Theorem, we need to assume that $\alpha$ satisfies some mild regularity condition (see \cite{goncar_convergence_1975} for the precise statement). 
In particular, the existence of a positive constant $c_0>0$ such that $\alpha(t)\geq c_0,\forall t$ a.e. is sufficient, and this is in fact guaranteed by Assumption \ref{as:measure_base}.

\begin{theorem}[Exponential convergence to atoms] \label{lemma:atom_convergence}
    Let $\mu$ satisfy Assumption \ref{as:measure_base} and Assumption \ref{as:single_interval}. Without loss of generality assume that $\supp \mu_{ac} = [-1,1]$. Let $\{P_n\}_{n\in\mathbb{N}}$ be a sequence of orthogonal polynomials of $\mu$, let
     \begin{align*}
        x_{1,n}<x_{2,n}<...<x_{n,n}
    \end{align*}
    be the zeros of $P_n$.
    Then the following statements hold:\begin{itemize}
        \item for all $n$ large enough, $P_n$ has exactly $r$ roots outside of $[-1,1]$,
        \item for all $n$ large enough, one may enumerate the roots outside of the continuous part of the support $[-1,1]$ 
        in such a way that $x_{j_n,n} \overset{n\to\infty}{\rightarrow} x_j$, where $x_j$ is an atom.
        \item fixing $j$ and considering the root sequence $(x_{j_n,n})_{n}$ that converges to $x_j$ as $n\rightarrow\infty$, there exists some $c_j< 1$ such that $$
            \limsup_{n\rightarrow\infty}|x_{j_n,n}-x_j|^{1/n}\leq c_j.
        $$
    \end{itemize}
\end{theorem}
\begin{proof}
    We prove the statement in section \ref{appendix:roots_convergence}.
\end{proof}

We follow with three theorems that work as a criterion for labeling roots as roots approximating atoms, roots approximating the continuous part or pollution.
The conditions that distinguish the roots become more complicated for general measures.

\begin{theorem} [Continuous and atomic part separation]\label{thm:distinguish}
    Suppose we have a measure $\mu$ satisfying Assumption \ref{as:measure_base} and Assumption \ref{as:single_interval}.
    Let $P_n$ be the monic polynomials orthogonal with respect to $\mu$.
    Let 
    \begin{align*}
        x_{1,n}<x_{2,n}<...<x_{n,n}
    \end{align*}
    be the zeros of $P_n$.
    Then, for any $\delta$ such that $0<2\delta <\Delta$, there exists $n_0 \in \mathbb{N}$, $n_0 = O(1/\delta)$ as $\delta \to 0+$, such that for all $n \geq n_0$, for every root $x_{i,n}$ the following holds:
    \begin{align*}
        \exists k \in \{ 1,2,...,n\}, k\not= i :  | x_{i,n} - x_{k,n} | < \delta&\Longleftrightarrow x_{i,n} \in \supp \mu_{ac} , 
        \\
        \forall j \in \{1,2,...,n\}, \ j \not= i : |x_{i,n} - x_{j,n}| \geq \delta
         &\Longleftrightarrow      x_{i,n} \in \supp \mu_{pp} + B(0,\delta).
    \end{align*}
    Moreover, for every point $y$ in $\supp \mu$, there exists $x_k,n$ such that
    \begin{align*}
        |y-x_{k,n}|<\delta.
    \end{align*}
\end{theorem}
\begin{proof}
    We prove the statement in section \ref{appendix:interval_atom_behavior}.
\end{proof}
We follow with a version of Theorem \ref{thm:distinguish} for measures with multiple continuous components as in Assumption \ref{as:atoms_outside}.
\begin{theorem}
\label{thm:distinguish_outside}
    Suppose we have a measure $\mu$ satisfying Assumption \ref{as:measure_base} and Assumption \ref{as:atoms_outside}.
    Let $\{P_n\}_n$ be the monic polynomials orthogonal with respect to $\mu$.
    Let 
    \begin{align*}
        x_{1,n}<x_{2,n}<...<x_{n,n}
    \end{align*}
    be the zeros of $P_n$.
    Then, for any $\delta$ such that $0<2\delta <\Delta$, there exists $n_0 \in \mathbb{N}$ (without a provable bound), such that for all $n \geq n_0$ and arbitrary $x_{i,n}$:
\begin{align*} 
     \exists k \in \{ 1,2,...,n\}, k\not= i :  | x_{i,n} - x_{k,n} | &< \delta
     \Longrightarrow 
     x_{i,n} \in \supp \mu_{ac} + B(0,\delta), 
        \\
        \forall j \in \{1,2,...,n\}, \ j\not = i : |x_{i,n} - x_{j,n}| \geq \delta
        \Longrightarrow
        &\left(
        x_{i,n} \in \supp \mu_{pp} + B(0,\delta) \right), \text{ or }
        \\
        &\left(x_{i,n} \in \mathrm{conv hull }\supp \mu_{ac} \setminus  \supp \mu_{ac}    \right).
\end{align*}
Moreover, for every point $y$ in $\supp \mu$, there exists $x_k,n$ such that
    \begin{align*}
        |y-x_{k,n}|<\delta.
    \end{align*}
\end{theorem}

\begin{proof}
    We prove the statement in section \ref{appendix:interval_atom_behavior}.
\end{proof}

Last is the most general situation, in which we only demand Assumption \ref{as:measure_base} and atoms can appear in between the intervals supporting the continuous part of the measure.
\begin{theorem}
\label{thm:distinguish_general}
    Suppose we have a measure $\mu$ satisfying Assumption \ref{as:measure_base}.
    Let $P_n$ be the monic polynomials orthogonal with respect to $\mu$.
    Let 
    \begin{align*}
        x_{1,n}<x_{2,n}<...<x_{n,n}
    \end{align*}
    be the zeros of $P_n$.
    Then, for any $\delta$ such that $0<2\delta <\Delta$, there exists $n_0 \in \mathbb{N}$  (without a provable bound) and $\rho>0$, such that for all $n \geq n_0$ and arbitrary $x_{i,n}$:
\begin{align*} 
     \exists k,l \in \{ 1,2,...,n\}, i\not=l\not=k\not= i :  (| x_{i,n} - x_{k,n}& | < \delta )\land
     (| x_{i,n} - x_{l,n} | < \delta )
     \Longrightarrow 
     x_{i,n} \in \supp \mu_{ac} + B(0,\delta), 
        \\
        \exists_1 k \in \{ 1,2,...,n\}, k\not= i :  (| x_{i,n} - x_{k,n} | < \delta )
     &\Longrightarrow 
     (x_{i,n} + x_{k,n})/2 \in \supp \mu_{pp} + B(0,\delta),
     \\
        \left(\forall j \in \{1,2,...,n\}, \ j\not = i : |x_{i,n} - x_{j,n}| \geq \delta \right)
        &\land 
        \left(\forall j \in \{1,...,n,n+1\}, \ j\not = i : |x_{i,n} - x_{j,n+1}| \geq \rho \right) \Longrightarrow
        \\
        &\Longrightarrow
        \left(
        x_{i,n}\not \in \supp \mu + B(0,\delta) \mathrm{, \ i.e. pollution}\right),
        \\
         \left(\forall j \in \{1,2,...,n\}, \ j\not = i : |x_{i,n} - x_{j,n}| \geq \delta \right)
        &\land 
         \exists k \in \{ 1,...,n,n+1\}, k\not= i :  (| x_{i,n} - x_{k,n+1} | < \rho )
         \Longrightarrow
         \\
         &\Longrightarrow
        \left(x_{i,n} \in \supp \mu_{pp}+ B(0,\delta)    \right).
\end{align*}
Moreover, for every point $y$ in $\supp \mu$, there exists $x_{k,n}$ such that
    \begin{align*}
        |y-x_{k,n}|<\delta.
    \end{align*}
\end{theorem}

\begin{proof}
    We prove the statement in section \ref{appendix:interval_atom_behavior}.
\end{proof}

\section{Multiplication operator and the truncated GNS }\label{sec:gen-curto-fialkow}

Theorem \ref{thm:flat-extension} has been the backbone of many algorithmic recovery methods \citep{HeltonMcCullough2004NCPositivstellensatz, LopezQuijorna2021DetectingGNS}.
The key tool is the truncated GNS construction   \citep{LopezQuijorna2021DetectingGNS}, which builds a certain operator, whose spectrum reveals information about the underlying measure.
We further interpret this perspective through the lens of spectral theory, which allows us to derive additional insights.
Before establishing the connection between the GNS construction used to recover a finitely atomic measure and the spectral properties of the multiplication operator on $L^2(\R,\mu)$, we recall some relevant facts and notation.

Throughout the paper, we consider the Hilbert space
\begin{align*}
    \mathcal{H} = L^2(\R,\mu)
\end{align*} with $\mu$ satisfying Assumption \ref{as:measure_base}.
For such $\mu$, we define the multiplication operator by the coordinate $x$ as
$M_x \in \mathcal{B}(\mathcal{H})$ as
\begin{align*}
    M_x f(x) = xf(x) \in L^2(\R,\mu).
\end{align*}
It is a standard result \citep{Kato} that the operator $M_x$ is self-adjoint.
A key fact, central to our purposes, is that the spectrum of $M_x$  coincides with the support of the measure 
\begin{align*}
    \sigma(M_x)= \text{ess Range}_{\mu} \ (x\mapsto x) =  \text{supp} (\mu), \\  \text{ where } x\mapsto x \text{ is the identity on } \R.
\end{align*}
This observation shows that finding
$\supp \mu$ is equivalent to determining the spectrum of the multiplication operator $M_x$ on $L^2(\R,\mu)$.

Let us briefly describe the spectrum of the multiplication operator in a detail. A priori, depending on the measure $\mu$,  the space $L^2(\mathbb{R},\mu)$ may be infinite dimensional and thus one must assume that
this the spectrum has two parts: a point spectrum and a continuous spectrum \citep{Simon}.
Under mild separation assumptions (the continuous and atomic parts have disjoint supports), these precisely coincide with the atomic and the continuous part of the support of $\mu$, that is
\begin{align*}
    \sigma(M_x)
    =
    \sigma_p(M_x) 
    \cup
    \sigma_c(M_x) 
    &,\quad
    \supp(\mu) = \supp(\mu_{pp}) 
    \cup
    \supp(\mu_{ac}),
    \\
    \sigma_p(M_x) = \supp(\mu_{pp}) ,
    &\quad
    \sigma_c(M_x)
    =
    \supp(\mu_{ac}).
\end{align*}

\subsection{Spectral Analysis of Finite-Dimensional Approximations}
From a computational perspective, even when the underlying problem is infinite dimensional, it must ultimately be treated in finite-dimensional terms. Our goal is to extract information about the spectrum of the (possibly infinite-dimensional) multiplication operator $M_x$ by studying suitable finite-dimensional approximations.

To this end, we employ a Rayleigh–Ritz (Galerkin-type) scheme. 
We consider an increasing sequence of finite-dimensional subspaces of $\mathcal H$ of the form
\begin{align*}
    \mathcal H_n
    =
    \operatorname{span}
    \{ [f_0]_\mu, [f_1]_\mu, \dots, [f_n]_\mu \}
    \subset \mathcal H,
\end{align*}
where $\{f_j\}_{j\ge0}$ is a chosen generating family in $L^2(\R,\mu)$ and $[\cdot]_\mu$ denotes the corresponding equivalence class.

We then study the compressed multiplication operators
\begin{align} \label{eq:truncated_multiplication}
    M_x^{(n)}
    :=
    \mathcal P_n M_x \mathcal P_n
    \big|_{\mathcal H_n} \in \mathcal{B}(\mathcal{H}_n),
\end{align}
where $\mathcal P_n$ denotes the orthogonal projection from $\mathcal H$ onto $\mathcal H_n$.
Thus $M_x^{(n)}$ is a finite-dimensional self-adjoint operator acting on $\mathcal H_n$.
The spectral properties of these compressed operators encode structural information about the underlying measure $\mu$.
The matrix representation of $M_x^{(n)}$ and the effectiveness of this method depend crucially on the choice of the generating family $\{f_j\}_{j\ge0}$.

Since our goal is to employ the spectral information of the sequence $(M_x^{(n)})_{n=1}^{\infty}$ in order to approximate the spectrum of $M_x$, it is natural to ask: what is the precise relationship between the spectra of the approximating operators and that of the original operator? In particular, we must determine whether the sequence is spectrally inclusive and whether it suffers from spectral pollution .

Below, we present definitions of these concepts that describe spectral approximation. The definitions themselves are specializations of the definitions of \cite{Bogli2017} to our setting.

\begin{definition}[Spectral inclusivity] \label{def:spectral_inclusivity}
Let $A, \{A_n\}_{n \in \mathbb{N}}$ be a bounded self-adjoint operators on a Hilbert space.
We say that the sequence $(A_n)$ is \emph{spectrally inclusive} of $A$ if for every $\lambda \in \sigma(A)$ there exists a sequence of elements $\lambda_n \in \sigma(A_n)$, $n \in \mathbb{N}$ with $\lambda_n \to \lambda$, that is, every point of the spectrum of $A$ is the limit of spectral points of $(A_n)_{n=1}^{\infty}$.
\end{definition}

\begin{definition}[Spectral pollution]
Let $A, \{A_n\}_{n \in \mathbb{N}}$ be a bounded self-adjoint operators on a Hilbert space $\mathcal{H}$. 
We say that an element $\lambda \in \R$ is a spurious eigenvalue if there exists an infinite subset $I \subset \mathbb{N}$ and $\lambda_n \in \sigma(A_n), n \in I$ with $\lambda_n \to \lambda$ but $\lambda \not \in \sigma(A)$.
The occurrence of such a point is known as spectral pollution.
\end{definition}

\begin{definition} [Spectral exactness] \label{def:spectral_exactness}
    Let $A, \{A_n\}_{n \in \mathbb{N}}$ be a bounded self-adjoint operators on a Hilbert space.
We say that the sequence $(A_n)$ is a \emph{spectrally exact} approximation of $A$ if it is \emph{spectrally inclusive} and no \emph{spectral pollution} occurs.
\end{definition}

As we mentioned above, the effectiveness of the Rayleigh-Ritz (Galerkin-type) method depends crucially on the choice of the generating family $\{f_j\}_{j\geq0}$.
In what follows, we provide a detailed insight in  on the situation when 
\begin{align} \label{eq:polynomial_subspace}
    f_j(x) = x^j, \quad \text{ and } \quad
    \mathcal H_n
    =
    \operatorname{span}
    \{ [1]_{\mu;\mathcal{H}}, [x]_{\mu;\mathcal{H}}, \dots, [x^n]_{\mu;\mathcal{H}} \}
    \subset \mathcal H.
\end{align}
One motivation for this choice is Theorem \ref{thm:spectrum+zeros}, which identifies the spectrum of the finite-dimensional compression with the roots of orthogonal polynomials:
\begin{align} \label{eq:determinant}
    \det ( xI - \Mxn) = P_n(x),
\end{align}
because the matrix $J_n$ is the operator $\Mxn$ expressed in a particular basis.
\begin{remark}
\label{remark:roots=spectrum}
    Due to the connection (\ref{eq:determinant}), the whole Section \ref{Sec:orthogonal_polynomials} with the orthogonal polynomial machinery then encapsulates arguments for studying $\Mxn$ with the underlying subspace chosen as in (\ref{eq:polynomial_subspace}).
\end{remark}
\begin{corollary}
    Due to (\ref{eq:determinant}), all the properties of root sequences \ref{def:root_sequence} stated in Theorems \ref{lemma:atom_convergence}, \ref{thm:distinguish}, \ref{thm:distinguish_outside} and \ref{thm:distinguish_general} are true for eigenvalues of the finite-dimensional approximation $\Mxn$.
\end{corollary}
In Section \ref{sec:OGPoly}, we leverage the orthogonal polynomials framework to show spectral exactness of $\Mxn$ for measures $\mu$ satisfying Assumption \ref{as:measure_base}.

\subsection{Truncated GNS Construction}

For the purpose of measure extraction, the truncated GNS construction has been  studied  extensively  by \cite{curto1996}, \cite{BurgdorfKlepPovh2016_OptimizationNoncommuting}, \cite{KlepPovhVolcic2018_MinimizerExtractionIsRobust}, \cite{LopezQuijorna2021DetectingGNS}.
Let us consider univariate polynomials $\R[x]_n$ of a fixed maximum  degree $n$.
Suppose we have
a bilinear form $\phi_n(\cdot,\cdot)$ over $\mathbb{R}[x]_{n}\times \mathbb{R}[x]_{n}$ induced by a measure $\mu$ as 
\begin{align} \label{eq:bilinear_form}
\phi_n(\cdot,\cdot):
    (p,q)\mapsto  \int p(x)q(x) d\mu(x), \quad p,q \in \R[x]_n.
\end{align}
Moreover, there is a one-to-one relationship between the bilinear form $\phi_n$ and the moment matrix $M_n(\mathbf{y})$ given by
\begin{align*}
    [M_n(\mathbf{y})]_{k,l} = \phi_n(x^k,x^l), \quad k,l=0,1,...,n.
\end{align*}
The truncated GNS construction gives us a matrix representation of a GNS multiplication operator (defined below in (\ref{eq:truncated_multiplication_explicit}))   acting on a particular vector space, whose construction is a part of the GNS.
We now describe the construction.

We obtain the desired $n$-th space for the truncated GNS multiplication operator by factorizing $\R[x]_n$ by
\begin{align*}
    \ker \phi_n(\cdot,\cdot) = \{ p \in \R[x]_n: \ \phi_n(p,p)=0 \}.
\end{align*}
The result is the following Hilbert space
\begin{align} \label{eq:GNS_space}
    \mathcal{K}_n =\R[x]_n\Big|_{\ker \phi_{n}}
\end{align}
with an inner-product induced by $\phi_n(\cdot,\cdot)$ (see \cite{LopezQuijorna2021DetectingGNS} for details).

\begin{theorem}
\label{thm:inclusions_and_isomorphisms}
Suppose that we have a sequence of bilinear forms $(\phi_n)_{n\geq1}$ or equivalently moment matrices $(M_n(\mathbf{y}))_{n\geq1}$.
In dimension one, exactly one of the following happens:      
    \begin{align}
    \label{eq:inclusions_flat}
    \text{1. } \quad & \mathcal{K}_1 \subsetneq \mathcal{K}_{2}
        \subsetneq...\subsetneq \mathcal{K}_{n_0} \simeq \mathcal{K}_{n_0+1}
        \simeq ... \simeq \mathcal{K}_{\infty}
        , \text{ or}
        \\ \label{eq:inclusions_non_flat}
         \text{2. } \quad &
        \mathcal{K}_1 \subsetneq \mathcal{K}_{2}
        \subsetneq 
        ...\subsetneq 
        \mathcal{K}_{n_0} \subsetneq \mathcal{K}_{n_0+1}
        \subsetneq ... \subsetneq\mathcal{K}_{\infty}
       .
    \end{align}
To be specific, (\ref{eq:inclusions_flat}) happens if and only if the flat extension condition (\ref{eq:flat-extension-prop}) holds for $n_0$ and 
(\ref{eq:inclusions_non_flat}) holds if the flat extension property (\ref{eq:flat-extension-prop}) is never satisfied.
\end{theorem}
\begin{proof}
    The proof is in Section \ref{proof:inclusions_and_isomorphisms}.
\end{proof}

The action of the GNS multiplication operator is
\begin{align} \label{eq:truncated_multiplication_explicit}
    M_{x;\text{GNS}}^{(n)} :
    \mathcal{K}_n 
    \longrightarrow \mathcal{K}_{n+1}
    \longrightarrow 
    \mathcal{K}_n
    :
    [p]_{\ker \phi_n} \mapsto  [xp]_{\ker \phi_{n+1}} \mapsto \Pi_n \left( [xp]_{\ker \phi_{n+1}} \right),
\end{align}
where $\Pi_n$ is:
\begin{itemize}
    \item the orthogonal projection from $\mathcal{K}_{n+1}$ onto $\mathcal{K}_n$ if $\mathcal{K}_n \subsetneq \mathcal{K}_{n+1}$, or
    \item
    $\iota^{-1}_n$,
    the inverse  of the canonical isomorphism $\iota_n$  between
    $\mathcal{K}_n$ and $\mathcal{K}_{n+1}$ (see \ref{eq:canonical_isomorphism} for precise definition) if $\mathcal{K}_n \simeq \mathcal{K}_{n+1}$.
\end{itemize}

We now rigorously describe the relationship between the GNS multiplication operator $M_{x;\text{GNS}}^{(n)}$ and the finite-dimensional compressions $M_x^{(n)}$ (\ref{eq:truncated_multiplication}) with the underlying subspace $\mathcal{K}_n$ as in (\ref{eq:polynomial_subspace}).
\begin{lemma}
\label{lemma:isomorphism}
    Suppose $\mu$ is a Borel measure on $\R$ which satisfies Assumption \ref{as:measure_base} and  induces the bilinear form $\phi_n$ (\ref{eq:bilinear_form}).
    Consider the three following sets
    \begin{align*}
    \mathcal{H}_n=
    \text{span } \{[1]_{\mu;\mathcal{H}}, [x]_{\mu;\mathcal{H}}, ..., [x^n]_{\mu;\mathcal{H}} \} \subset L^2(\R,\mu), \quad \R[x]_n \Big|_{\mu}, \quad 
        \mathcal{K}_n =\R[x]_n \Big|_{\ker \phi_n}.
    \end{align*}
    Then 
    \begin{align} \label{eq:set_equality}
    \text{1.} \quad &
        \R[x]_n \Big|_{\mu }
        =
        \R [x]_n \Big|_{\ker \phi_n} = \mathcal{K}_n,
    \\\label{eq:Hilbert_space_equality}
    \text{2.}\quad  &
        \mathcal{K}_n=
        \Big(
            \R[x]_n \Big|_{\mu} ; \phi_n(\cdot,\cdot) \Big) \simeq 
            \Big(
            \text{span } \{[1]_\mu, [x]_\mu, ..., [x^n]_\mu \} ; \langle \cdot, \cdot \rangle_{L^2(\R,\mu)}\Big) = \mathcal{H}_n,   \end{align}
            i.e. $\mathcal{K}_n$ and $\mathcal{H}_n$ are isometrically isomorphic.
\end{lemma}
\begin{proof}
    The proof is in Section \ref{proof:isomorphisms}.
\end{proof}

\begin{theorem} \label{thm:unitary_equivalence}
    The operators $\Mxn$ and  $M_{x;\text{GNS}}^{(n)}$ are unitarily equivalent via an operator
    \begin{align*}
    T:
\mathbb{R}[x]_n\big|_\mu
\longrightarrow
\mathcal{H}_n:[p]_{\mu;\mathbb{R}[x]_n} \mapsto[p]_{\mu;L^2(\R,\mu)},
\end{align*}
    i.e. $T \Mxn = M_{x;\text{GNS}}^{(n)}T$, which results into the following commuting diagram:
    \[
\begin{tikzcd}
\mathcal K_n 
\arrow[r,"M_{x;\text{GNS}}^{(n)}"] 
\arrow[d,"T"']
&
\mathcal K_n 
\arrow[d,"T"]
\\
\mathcal H_n
\arrow[r,"\Mxn"]
&
\mathcal H_n
\end{tikzcd}
\]
\end{theorem}
\begin{proof}
    The proof is a result of the fact that $T$ is an isometric isomorhipsm (as we show in the proof of Lemma \ref{lemma:isomorphism}) and that
    \begin{align*}
        M_x^{(n)}= T^{-1} M_{x;\text{GNS}}^{(n)} T.
    \end{align*}
\end{proof}

\begin{remark}
    Since $M_x^{(n)}$ and $M_{x;\text{GNS}}^{(n)}$ are unitarily equivalent, they have identical spectral properties.
\end{remark}

\begin{remark}
\label{remark:TSSOS}
    The robust extraction method of TSSOS \cite{wang2021tssos} actually implements the procedure described in the current section. The library TSSOS implements the robust GNS procedure of \cite{KlepPovhVolcic2018_MinimizerExtractionIsRobust} through the function \verb|extract_solutions_robust|.  When we apply this function to a positive definite moment matrix (of a measure that is supported over infinitely many points), we obtain the matrix representation of the operations : \begin{equation}
        \mathbb{R}[x]_{\leq d}\xrightarrow[]{x_i.}\mathbb{R}[x]_{\leq d+1}\xrightarrow[]{\pi_\phi}\mathbb{R}[x]_{\leq d}
    \end{equation}
    in the basis spanned by orthogonal polynomials. Furthermore, in the univariate situation, the multiplication matrix there obtained is equivalent to the Jacobi matrix that we have introduced in the beginning of section \ref{Sec:orthogonal_polynomials}. The multiplication matrix given by the GNS procedure will produce the roots of the orthogonal polynomials in 1D.
\end{remark}


\subsection{The spectrum of $M_x$ when $\mu_{ac} = 0$ }
\label{Sec:atomic_measure}
The observations above hold if the bilinear form $\phi_n$ is generated by an arbitrary measure satisfying assumption \ref{as:measure_base}.
Now, we further discuss the properties of $\mathcal{H}_n$ (\ref{eq:polynomial_subspace}), $\mathcal{K}_n$ (\ref{eq:GNS_space}) in case when the measure is finitely atomic, i.e. $\mu_{ac} = 0$.
Finitely atomic measure implies that the flatness condition (\ref{eq:flat-extension-prop}) must eventually be satisfied for high enough $n$.
\cite[Proposition 5.16]{LopezQuijorna2021DetectingGNS}
states the following equivalence:
\begin{align*}
    \text{ The flatness condition } (\ref{eq:flat-extension-prop}) \text{ is satisfied } 
    \Longleftrightarrow
\mathcal{K}_n \simeq \R^r \simeq \mathcal{K}_{n+1}. 
\end{align*}
This means that flatness condition implies that the spaces $\mathcal{K}_n, \mathcal{K}_{n+1}$ in the definition of $ M_{x;\text{GNS}}^{(n)}$ in (\ref{eq:truncated_multiplication_explicit}) are isomorphic and therefore, we do not lose any spectral information by applying $\Pi_n$ to $[xp]_{\ker \phi_{n+1}}$. Furthermore, the unitary equivalence \ref{thm:unitary_equivalence} implies that under the flatness assumption (\ref{eq:flat-extension-prop}), 
$M_x^{(n)}$ can be expressed as a $r\times r$ matrix as the
following series of isomorphisms hold
\begin{align} \label{eq:isomorphisms}
 \mathcal{K}_{n+1} \simeq \mathcal{K}_n 
    = \R[x]_n \Big|_{\mu \text{ equal a.e.}}
    \simeq \mathcal{H}_n = \mathcal{H} = L^2(\R,\mu) \simeq \R^r,
\end{align}
which technically allows us to accurately extract the whole spectrum of the  operator $M_x$ as it can be represented using a finite dimensional basis. 

We conclude with the observation that in the case of an atomic measure, the sequence of restricted multiplication operators $M_x^{(n)}$ is  a \emph{spectrally exact} \ref{def:spectral_exactness} approximation of the underlying operator $M_x$ on $L^2(\R,\mu)$ and that the approximating sequence of eigenvalues $(\lambda_n)_n$ in the sense of definition \ref{def:spectral_inclusivity} stabilizes after finite amount of steps.

\subsection{The spectrum of $M_x$ when $\mu_{ac} \neq 0$.}
\label{sec:OGPoly}

Assume that the measure $\mu$ from assumption \ref{as:measure_base} satisfies $\mu_{ac} \neq 0$  or equivalently the spectrum of $M_x$ acting on $L^2(\R,\mu)$ has a non-zero continuous and  discrete part. 
Under this assumption, the domain $\mathcal{H}_n$ of the approximation operator $\Mxn$ (or equivalently the domain $\mathcal{K}_n$ of the GNS multiplication operator $M_{x;\text{GNS}}^{(n)}$) never becomes isomorphic to the whole space $\mathcal{H}$ as in (\ref{eq:isomorphisms}). Therefore, for every $n\in \mathbb{N}$, the operator $\Mxn$
represents only a finite-dimensional approximation of the multiplication operator $M_x$.

This is precisely the framework in which one can apply the classical results concerning the zeros of orthogonal polynomials associated with measures of infinite support. In this setting, the orthogonal polynomial sequence $\{P_n\}_{n\geq0}$ tied to the space $L^2(\R,\mu)$ is well-defined and non-degenerate, allowing the full machinery of zero distribution theory (see Section \ref{Sec:orthogonal_polynomials}) to be employed in the sense of Remark \ref{remark:roots=spectrum}.

Concerning \emph{spectral exactness}, point \ref{thm:density_of_zeros} of
Theorem \ref{thm:properties_og_poly} alone establishes \emph{spectral inclusivity}.
Theorem \ref{lemma:atom_convergence}  ensures there is no \emph{spectral pollution} as  $\Mxn$ has precisely $r$ eigenvalues outside of $\supp \mu_{ac}$,  corresponding to the $r$ atoms and they converge to the atoms exponentially quickly as $ n \to \infty$.
\begin{theorem}[Spectral exactness] \label{thm:spectral_exactness}
    Suppose $\mu$
    satisfies Assumption \ref{as:measure_base}.
    Then, the sequence of finite-dimensional truncations $M_x^{(n)}$ of $M_x$ is \emph{spectrally inclusive} and we can identify the pollution.
    Moreover, if $\mu$ in addition satisfies Assumption \ref{as:single_interval}, $M_x^{(n)}$ is \emph{spectrally exact}.
\end{theorem}
\begin{proof}
    Since we know that the spectrum of $M_x^{(n)}$ corresponds to the zeros of orthogonal polynomials, this is a direct corollary of
    Theorems 
    \ref{thm:properties_og_poly}, \ref{lemma:atom_convergence}, \ref{thm:distinguish}, \ref{thm:distinguish_outside} and \ref{thm:distinguish_general}.
\end{proof}

\section{Numerical illustrations}\label{sec:num}

To corroborate our results of Section \ref{sec:results}, we present 
results of numerical experiments with Algorithm~\ref{alg:atom_extraction_theoretic} in two settings.
In the first set of experiments, there is a single interval
in $\mu_{ac}$ and multiple points in $\mu_{pp}$, as foreseen by Assumption~\ref{as:single_interval}.
In the second set of experiments, there are two disjoint intervals in 
$\mu_{ac}$ and multiple points outside of these in $\mu_{pp}$, which is covered by  Assumption~\ref{as:atoms_outside}.
Recall that Theorem~\ref{thm:algorithm_proof} establishes that Algorithm~\ref{alg:atom_extraction_theoretic} solves Problem~\ref{prob:reconstr} in finite time under Assumption~\ref{as:measure_base} and in polynomial time when Assumption~\ref{as:single_interval} additionally holds, 
assuming  exact moments of the underlying measure $\mu$ are known.
The performance of Algorithm~\ref{alg:atom_extraction_theoretic} is governed by five parameters: the number of atoms and continuous components (fixed at one atom and one versus two intervals throughout); the separation distance $\Delta > 0$; the target precision $\varepsilon > 0$; the length $r > 0$ of the intervals supporting the continuous component; and the center $c \in \mathbb{R}$ of those intervals.
Across both settings, we provide a systematic study of the performance of atom and interval recovery as functions of $(a, r)$.

\paragraph{The procedure}
Both sets of experiments proceed in three steps.
First, we solve the moment relaxation \citep{Lasserre2001GlobalOptimization} of a polynomial optimization problem (POP)~(\ref{eq:experiments_2}) or
(\ref{eq:experiments_3})
at orders $N \in \{4, 5, \ldots, 9\}$, obtaining pseudomoment matrices $M_N$.
Second, we assume $M_N$ 
satisfies Assumption~\ref{as:measure_base} and run Algorithm~\ref{alg:atom_extraction_theoretic} with $\varepsilon = 10^{-2}$.
Third, we measure the accuracy of estimating $\supp \mu$. 
Given $(a, c, r, N)$, atom recovery is deemed successful if Algorithm~\ref{alg:atom_extraction_theoretic} produces exactly one candidate atom in $\mathcal{A}_N$ within $\varepsilon$ of the true atom $a+c+r$.
We measure the fidelity of the recovered continuous support via the intersection-over-union (IoU) metric. 
Let us note that for a fixed  $(a,r,N)$, we compute the success rate as an average over a collection of parameters $c$ (the problem is not translationally invariant).

\paragraph{One interval}
To test the behaviour of Algorithm~\ref{alg:atom_extraction_theoretic} under Assumption~\ref{as:single_interval}, we consider the POP: 
\begin{align}
    \min\ &1 \nonumber \\
    \text{s.t.}\quad &R^2 - x^2 \geq 0, \nonumber \\
    &-(x - (c-r))(x - (a+c+r)) \geq 0, \label{eq:experiments_2} \\
    &(x-(c+r))(x-(a+c+r)) \geq 0. \nonumber 
\end{align}
whose representing measure satisfies:
\[
    \operatorname{supp}\mu_{ac} = \bigl[c-r,\;c+r\bigr],\qquad
    \operatorname{supp}\mu_{pp} = \{a+c+r\},
\]
considering that the support is the feasible set of (\ref{eq:experiments_2}). The separation distance is clearly 
\[
    \Delta = a.
\]

\paragraph{Two intervals}
To test the behaviour of Algorithm~\ref{alg:atom_extraction_theoretic}  under Assumption~\ref{as:measure_base}, we consider the POP: 
\begin{align}
    \min\ &1 \nonumber \\
    \text{s.t.}\quad &R^2 - x^2 \geq 0, \nonumber \\
    &-(x - (c-r))(x - (a+c+r)) \geq 0, \label{eq:experiments_3} \\
    &(x-(c+r))(x-(a+c+r)) \geq 0, \nonumber \\
    &(x-(c-r/3))(x-(c+r/3)) \geq 0. \nonumber
\end{align}
whose representing measure $\mu = \mu_{ac;1} +  \mu_{ac;2} +\mu_{pp}$ satisfies Assumption~\ref{as:atoms_outside}, with
\[
    \operatorname{supp}\mu_{ac;1} = \bigl[c-r,\;c-\tfrac{r}{3}\bigr],\qquad
    \operatorname{supp}\mu_{ac;2} = \bigl[c+\tfrac{r}{3},\;c+r\bigr],\qquad
    \operatorname{supp}\mu_{pp} = \{a+c+r\}.
\]

\begin{figure}[H]
\label{FIG:MAIN}
    \centering

    \begin{subfigure}[t]{0.48\linewidth}
        \centering
        \includegraphics[width=\linewidth]{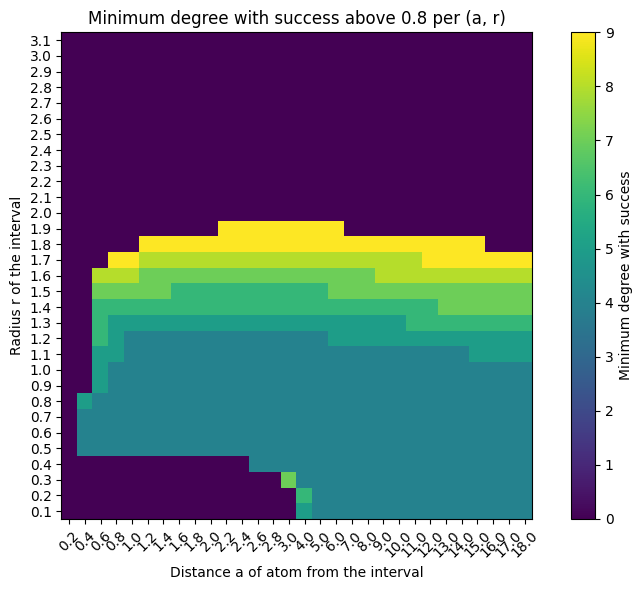}
        \caption{\textbf{One interval}: Minimum degree needed to identify the atom with success rate above 80 \%. }
        \label{fig:sub1}
    \end{subfigure}
    \hfill
    \begin{subfigure}[t]{0.48\linewidth}
        \centering
        \includegraphics[width=\linewidth]{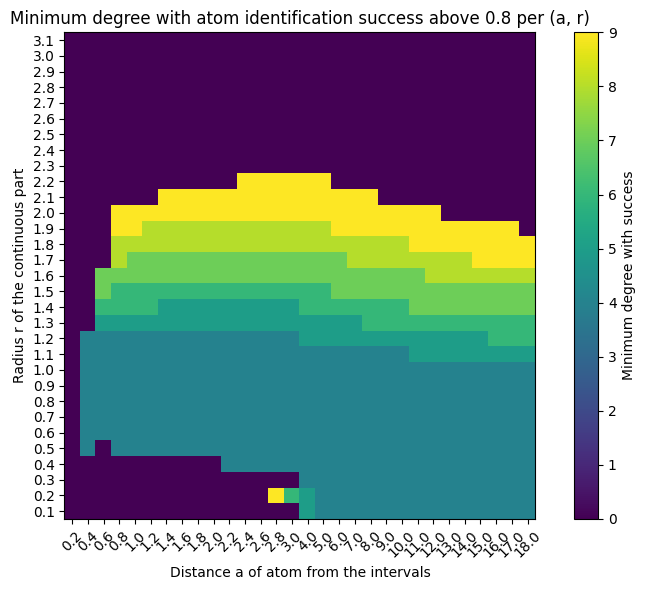}
        \caption{\textbf{Two intervals}: Minimum degree needed to identify the atom with success rate above 80 \%.}
        \label{fig:sub2}
    \end{subfigure}

    \vspace{0.5em}

    \begin{subfigure}[t]{0.48\linewidth}
        \centering
        \includegraphics[width=\linewidth]{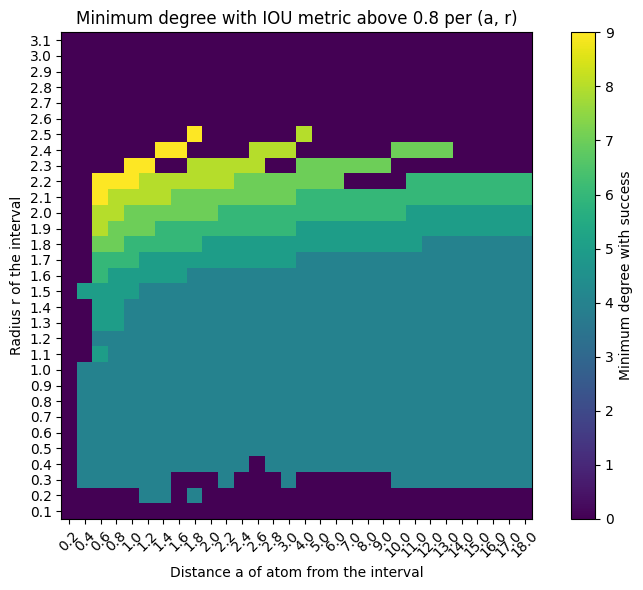}
        \caption{\textbf{One interval}: Minimum degree needed to identify the interval with IoU metric above 80 \%.}
        \label{fig:sub3}
    \end{subfigure}
    \hfill
    \begin{subfigure}[t]{0.48\linewidth}
        \centering
        \includegraphics[width=\linewidth]{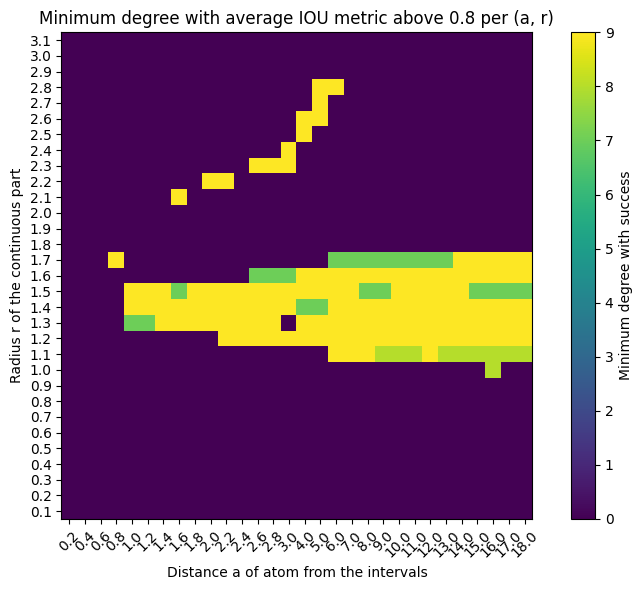}
        \caption{\textbf{Two intervals}: Minimum degree needed to identify the interval with IoU metric above 80 \%.}
        \label{fig:sub4}
    \end{subfigure}

    \vspace{0.5em}

    \begin{subfigure}[t]{0.48\linewidth}
        \centering
        \includegraphics[width=\linewidth]{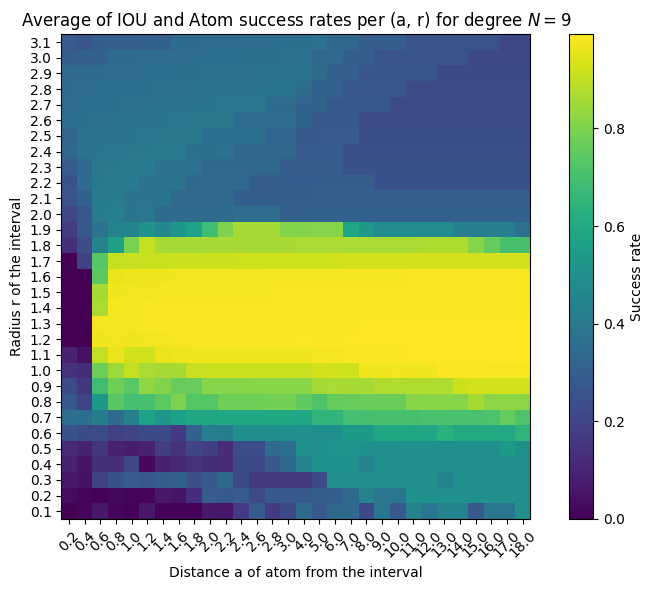}
        \caption{
        \textbf{One interval}: Average of IoU and atom identification success rate for fixed $N=9$.}
        \label{fig:sub5}
    \end{subfigure}
    \hfill
    \begin{subfigure}[t]{0.48\linewidth}
        \centering
        \includegraphics[width=\linewidth]{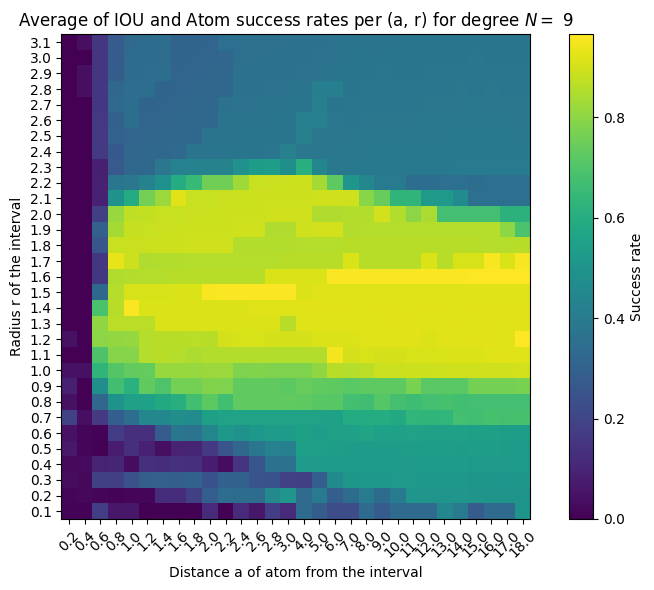}
        \caption{\textbf{Two intervals}: Average of IoU and atom identification success rate for fixed $N=9$.}
        \label{fig:sub6}
    \end{subfigure}

    \caption{Support recovery comparison of the one interval and two interval setting.}
    \label{fig:locatom-success}
\end{figure}
In the two interval settings, we again assume that the support of $\mu$ (which we aim to recover) is the feasible set of (\ref{eq:experiments_3}). The separation distance for the two interval case  is clearly 
\[
    \Delta = \min\;\left\{{\textstyle\frac{2r}{3}},\,a\right\}.
\]



Figure \ref{fig:locatom-success} presents the results in single-interval (Assumption \ref{as:single_interval}) and two-intervals (Assumption \ref{as:atoms_outside}) settings for certain choices of parameters $(a,r)$.
In particular, the figure presents the minimum degree $N$ needed in order to reach certain success rate of identifying the atoms (Figures \ref{fig:sub1}, \ref{fig:sub2})
and of identifying the interval(s) (Figures \ref{fig:sub3}, \ref{fig:sub4}).
The value of 0 (displayed in dark blue) indicates that the threshold was not reached for any degree $N=4,5,...,9$.
In Figures \ref{fig:sub5}, \ref{fig:sub6}, we fix the degree to $N=9$ and plot the average between the atomic identification success rate and IoU displaying the performance of identifying $\supp \mu = \supp \mu_{ac} + \mu_{pp}$.

Notice that solving the semidefinite programming relaxation of either (\ref{eq:experiments_2}) or (\ref{eq:experiments_3}) and the root-extraction are numerically challenging for $N > 9$.
In theory, this restricts the attainable precision to $\varepsilon \approx 1/N \approx 0.1$ even under the idealized conditions of Theorem~\ref{thm:algorithm_proof} with Assumption \ref{as:single_interval} satisfied, when we consider with the range $N \in \{4,\ldots,9\}$, rather than $N > 1/\varepsilon$. 
Furthermore, when the continuous part of the representing measure is supported on two disjoint intervals in (\ref{eq:experiments_3}), the exponential convergence guarantee of Theorem~\ref{lemma:atom_convergence} does not apply and spectral pollution may occur.
Still, as Figure~\ref{fig:locatom-success} shows, the cases where $N = 9$ does not suffice (displayed in dark blue) are rare.

Despite these limitations, our experiments reveal two consistent empirical phenomena.
The root sequence associated with each atom exhibits convergence that is numerically consistent with exponential decay for both one-interval and two-interval settings -- corroborating Theorem~\ref{lemma:atom_convergence}. For a certain non-negligible range of parameter values, relaxation orders $n \in \{5,\ldots,9\}$ suffice to isolate individual atoms and to recover one and even two continuous components with acceptable accuracy.

\paragraph{Additional experimental results}

Appendix~\ref{Sec:TSSOS_vs_OGPoly} provides additional numerical evidence for Remark~\ref{remark:TSSOS}: the root-extraction step of Algorithm~\ref{alg:atom_extraction_theoretic} agrees with the output of routine \verb|extract_solutions_robust| of TSSOS \citep{wang2021tssos} to machine-precision accuracy. 
In particular, Appendix~\ref{Sec:TSSOS_vs_OGPoly} reports two further sets of experiments. 
Figure~\ref{fig:two_polynomials} in that appendix illustrates the behavior of two consecutive orthogonal polynomials $p_8$ and $p_9$ generated by a representative pseudomoment matrix, confirming the interlacing property (Theorem~\ref{thm:interlacing}), the convergence of roots toward the atom (Theorem~\ref{lemma:atom_convergence}), and the containment of roots within the convex hull of $\operatorname{supp}\mu$ (Theorem~\ref{thm:roots_and_support}).
Two spectral-pollution roots are visible inside the gap $[c-r/3,\,c+r/3]$; their presence is consistent with the theory, which permits at most one polluting root in a measure-zero set (Theorem~\ref{thm:properties_og_poly}, item~\ref{thm:measure_0_sets}).
The second set of experiments (Table~\ref{tab:tssos_comparison} and Figure~\ref{fig:placeholder}) validates Remark~\ref{remark:TSSOS}: across 120 problem instances, the outputs of Algorithm~\ref{alg:roots_from_moments} and \verb|extract_solutions_robust| (with \textbf{rtol} set to zero) agree to within $3.41 \times 10^{-11}$ in cumulative absolute difference, confirming that the TSSOS routine implements the same root-finding procedure as our algorithm up to machine-precision errors.

\section{Conclusion}
We studied the problem of reconstructing the support of a measure 
from a finite number of its moments. Under mild assumptions on the measure, we proved 
asymptotic results for the zeros of orthogonal polynomials that allow us to recover the support and distinguish between the continuous and atomic parts of the measure.
These theoretical results lead to a finite-time extraction algorithm, which is polynomial-time under additional assumptions. We tested the
algorithm on moment/sum-of-squares relaxations of  polynomial optimization problems, where the minimizer
extraction beyond finitely-atomic representing measures had been a major open problem. The numerical results are consistent with our theoretical results and, in several cases, seem to benefit from rates of convergence faster than those proven. Improving the bounds on the convergence rates of Algorithm~\ref{alg:atom_extraction_theoretic} may therefore be an important direction for future research.   

\section*{Acknowledgements}
S.M. and J.M. were funded by the European Union under the project
ROBOPROX (reg. no. CZ.02.01.01/00/22 008/0004590).
R. K.  and A. W. were in part funded by the Czech Science Foundation (23-07947S). 

\bibliographystyle{siamplain}
\bibliography{references}

\newpage
\appendix

\section{Proofs}\label{sec:proof-section}

\subsection{Proof of Theorem \ref{lemma:atom_convergence}}
\label{appendix:roots_convergence}
We assume that $\text{supp}\mu_{ac}=[-1,1]$. Due to Assumption \ref{as:measure_base}, Szego's condition holds for the density, see \citep{goncar_convergence_1975}. In particular the Szego's condition is implied by the density function being lower bounded by a uniform constant $c_0>0$.

In this section, we state some important facts given by \citep{goncar_convergence_1975}, on which we build our results. 
Consider a measure $\nu$, compactly supported on $[-1,1]$ and absolutely continuous with respect to the Lebesgue measure. 
Define the Stieltjes transform of $\nu$, which is a function $g$, holomorphic outside of $[-1,1]\subset\mathbb{C}$, and is defined by the integral transform  \begin{equation}g(z)=\int_{[-1,1]}\frac{\mathrm{d}\nu(t)}{z-t},\forall z\notin[-1,1].
\end{equation}
The function $g$ admits a Laurent series decomposition at $z=\infty$, indeed, for all $z$ in a neighborhood of $\infty$, we compute \begin{align*}
    g(z)&=z^{-1}\sum_{n\geq0}\int_{-1}^{1}(t/z)^n\mathrm{d}\nu(t)\\
    &=z^{-1}\sum_{n\geq0}g_nz^{-n}
\end{align*}
where $g_n=\int_{-1}^{1}t^n\mathrm{d}\nu(t),\forall n\geq 0$ constitutes the moment sequence of $\nu$. Therefore, the Stieltjes transform is a variant of the generating function of the moment sequence. We also define the degree-$n$ Padé approximation of $g$, which consists of a pair of polynomials $(p_n[g](z),q_n[g](z))$ such that $q_n[g]=z^n+\dots$ has degree $n$, is monic and the following identity holds : \begin{equation}\label{eqn:pade_g}
    g(z)-\frac{p_n[g](z)}{q_n[g](z)}=\frac{C_n}{z^{2n+1}}+\dots
\end{equation}
where $C_n\in\mathbb{C}$, whose value is unimportant because it is only required that $p_n[g]/q_n[g]$ produces correctly the first $0,\dots,2n-1$ moments of $\nu$.

From (\ref{eqn:pade_g}), a simple multiplication yields \begin{equation}\label{eqn:pade_g_2}
    q_n[g](z)g(z)-p_n[g](z)=\frac{\tilde{C}_n}{z^{n+1}}+\dots
\end{equation}
By comparing the coefficient of the power $z^k$ with $k=-1,\dots,-n$ at both side of (\ref{eqn:pade_g_2}), one obtains a linear system of $n$ equations which allows to find $q_n[g]$. Define $G_n\in\mathbb{C}^{(n+1)\times(n+1)}$ where $[G_{n}]_{ij}=g_{i+j},\forall i,j=0,\dots,n$, then \begin{equation}\label{eqn:pade_g_denominator}
    q_n[g](z)=\frac{1}{\det G_{n-1}}\begin{pmatrix}
        g_0&g_1&\dots &g_n\\
        g_1&g_2&\dots &g_{n+1}\\
        \vdots&\vdots&\vdots\\
        g_{n-1}&g_{n}&\dots&g_{2n-1}\\
        1&z&\dots&z^{n}
    \end{pmatrix}
\end{equation}
by the Cramer's rule for solving linear system, if $\det G_{n-1}\neq 0$. Since $\nu$ is a positive measure whose support is infinite, $q_n[g]$ is always well defined because $\det G_n\neq 0,\forall n$. We should also realize that the denominators of the Padé approximations of $g(z)$ are precisely the monic orthogonal polynomials of $\nu$, through the determinantal characterization. 
If one manages to prove, for example, uniform convergence of $p_n[g]/q_n[g]$ towards $g$ as $n\rightarrow\infty$, we can obtain information about the distribution of the roots of $q_n[g](z)$, which at limit should approach the singularities of $g(z)$, meaning the interval $[-1,1]$.

Gonchar proved a result of this type in 1975, which is the cornerstone that we use to prove Theorem \ref{lemma:atom_convergence} that we use in the current study. Consider a measure $\mu_{ac}$ supported over $[-1,1]$ and absolutely continuous with respect to the Lebesgue measure. Let $g$ be its Stieltjes transform. We consider then $x_j\notin[-1,1],j=1,\dots,l$, integer multiplicity $m_j\geq 1$ and complex weights $A_{j,k}\in\mathbb{C},\forall j=1,\dots,l,\forall k=1,\dots,m_j$ with $A_{j,m_j}\neq 0$, we define the rational function \begin{equation}
    r(z)=\sum_{j=1}^l\sum_{k=1}^{m_j}\frac{A_{j,k}}{(k-1)!(z-x_j)^k}.
\end{equation}
and we perturb the Stieltjes transform $g$ by adding $r$ to it : \begin{equation}
    f=g+r=z^{-1}\sum_{n\geq 0}f_nz^{-n}
\end{equation}
It is easy to notice that the singularity points of $f$ are the interval $[-1,1]$ and the isolated singularities $x_1,\dots,x_l$ of multiplicity $m_1,\dots,m_l$. Gonchar proved a uniform convergence result of the Padé approximation $p_n[f]/q_n[f]$ towards $f$. Because of the perturbation $r$, it is not necessarily true that the Padé approximation exists for all $n$. Define $F_n\in\mathbb{C}^{(n+1)\times(n+1)}$, where $[F_n]_{i,j}=f_{i+j}, \ \forall i,j = 0,...,n$, and $\Lambda=\{n\in\mathbb{N},\det F_{n-1}\neq 0\}$, then from the determinantal formula (\ref{eqn:pade_g_denominator}) $q_n[f]$ exists for all $n\in\Lambda$. Gonchar proved the following theorem :

\begin{theorem}[Convergence of Padé approximation,\cite{goncar_convergence_1975}]
\label{thm:Gonchar_stuff}
    Consider an absolutely continuous measure $\mu_{ac}$
    with $\supp\mu_{ac}=[-1,1]$ and assume that $\mu_{ac}$ satisfies the Szego's condition. Construct $f=g+r$ and assume that $\Lambda$ is infinite. Let $p_n/q_n$ be the Padé approximation of $f$ for all $n\in\Lambda$, then: \begin{itemize}
        \item For all open set $U\subset \mathbb{C}\setminus[-1,1]$, there exists $n(U)$ so that for all $n\geq n(U),n\in\Lambda$, the number of roots of $q_n$ in $U$ is equal to the number of poles of $r$, counting multiplicity.
        \item For all compact sets $K\subset\mathbb{C}\setminus[-1,1]$ such that $x_j\notin K,\forall j=1,\dots,l$, there exists a constant $\rho(K)>1$ such that $$
            \limsup_{n\rightarrow\infty,n\in\Lambda}\norm{f-\frac{p_n}{q_n}}_{K}^{1/n}\leq 1/\rho(K)<1.
        $$
    \end{itemize}
\end{theorem}
Now we present the proof of Theorem \ref{lemma:atom_convergence}. The proof is not new, the same techniques are used also in \cite{goncar_convergence_1975}, notably those used in the proof of the last point.

\begin{proof}[Proof of Theorem \ref{lemma:atom_convergence}]
We take the special case of real $x_j\notin[-1,1]$, $m_j=1,\forall j=1,\dots,l$, $A_{j,1}=\alpha_j$ and realize that in this situation \begin{equation}
    f=g+r=\int_{\mathbb{R}}\frac{\mathrm{d}\mu(t)}{z-t}
\end{equation}
where $\mu=\mu_{ac}+\mu_{pp},\mu_{pp}=\sum_{j=1}^l\alpha_j\delta_{x_j}$. Therefore $f$ is the generating function of the moment sequence of $\mu$, from the determinantal characterization of the monic orthogonal polynomials, $q_n$'s are the monic orthogonal polynomials of $\mu$. Since the support of $\mu_{ac}$ is infinite and $\mu$ is a positive measure, we conclude that $\det F_n>0,\forall n\in\mathbb{N}$ therefore $\Lambda=\mathbb{N}$ is infinite.

To prove the first point of Theorem \ref{lemma:atom_convergence}, we consider the open set $U=\mathbb{C}\setminus[-1,1]$. For all $n$ large enough, $q_n$ has exactly $r$ roots in $U$ due to Theorem \ref{thm:Gonchar_stuff}.

To prove the second point, for each atom $x_j$, we consider a sequence of open balls $B(x_j,1/n),\forall n\in \mathbb{N}$. For all $\epsilon >0$, there exists $N_1$ large enough so that $1/N_1<\epsilon$. Now take $N_2=\max\{n(B(x_j,1/N_1)),N_1\}$, from Theorem \ref{thm:Gonchar_stuff}, for all $n>N_2$, $q_n$ has precisely one root $x_{j_n,n}$ in $B(x_j,1/N_1)$ and therefore $\vert x_{j_n,n} - x_j \rvert<1/N_1<\epsilon,\forall n>N_2$. $\epsilon$ being arbitrary, this proves the second point of Theorem \ref{lemma:atom_convergence}.

Finally we prove the exponential convergence. For all $x_j$, one can find a small $\delta>0$ so that the closed ball $\overline{B}(x_j,\delta)$ does not intersect $[-1,1]\cup\{x_i,i\neq j\}$. For all $n$ large enough, $q_n$ has precisely one root in the open ball $B(x_j,\delta)$. Since $\partial B(x_j,\delta)$ is compact, second point of
Theorem \ref{thm:Gonchar_stuff} states the following.
For all $\epsilon>0$ there exists some $N$ large enough so that for all $n>N$, $\norm{f(z)-\frac{p_n(z)}{q_n(z)}}_{\partial B(x_j,\delta)}^{1/n}\leq \frac{1+\epsilon}{\rho(\partial B(x_j,\delta))}$. We obtain the following chain of inequalities (here $x_{j_n,n}$ is the unique root of $q_n$ that is inside $B(x_j,\delta)$)
\begin{align}\label{proof-thrm:estimate1}
    &\norm{f(z)(z-x_j)(z-x_{j_n,n})-\frac{p_n(z)}{q_n(z)}(z-x_j)(z-x_{j_n,n})}_{L^\infty(\partial B(x_j,\delta))}^{1/n}\\
    \leq&\norm{f(z)-\frac{p_n(z)}{q_n(z)}}_{\partial B(x_j,\delta)}^{1/n}2^{1/n}\delta^{2/n}\\
    \leq&\frac{1+\epsilon}{\rho(\partial B(x_j,\delta))}2^{1/n}\delta^{2/n}.
\end{align}
We realize that the function under the norm symbol of (\ref{proof-thrm:estimate1}) is holomorphic in $B(x_j,\delta)$ eventually for large enough $n$, as a result, the supremum of $\norm{.}_{\partial B(x_j,\delta)}$ is equal to $\norm{.}_{\overline{B}(x_j,\delta)}$. Now evaluate the function under the norm at $z=x_j$ and note $\alpha_j=\lim_{z\rightarrow x_j}f(z)(z-x_j)\neq 0$ (this is because $x_j$ is a simple pole of $f$), we obtain 
\begin{equation}
    |x_j-x_{j_n,j}|^{1/n}\leq\left(\frac{2\delta^2}{\alpha_j}\right)^{1/n}\frac{1+\epsilon}{\rho(\partial B(x_j,\delta))}
\end{equation}
from which we deduce that for 
$\delta$ sufficiently small, $$
    \limsup_{n\rightarrow\infty}|x_j-x_{n,j}|^{1/n}\leq \frac{1}{\rho(\partial B(x_j,\delta))}<1,
$$ which concludes the exponential convergence.
\end{proof}

\subsection{Proof of Theorems \ref{thm:distinguish}, \ref{thm:distinguish_outside} and \ref{thm:distinguish_general}}
\label{appendix:interval_atom_behavior}

We aim to leverage the fact that polynomials orthogonal w.r.t. a measure from a certain class of   "regular measures" have convenient properties in terms of behavior of their zeros on $\supp \mu_{ac}$.
It is trivial to show that the measures satisfying Assumption \ref{as:measure_base}, that we consider, are regular.

We start by introducing several definitions from measure potential theory \cite{Simon2007_equilibrium}, \cite{Widom1967}, \cite{SaffTotik1997}.

\begin{definition}
Let $E \subset \mathbb{R}$ be a nonempty compact set, and let $\mathcal{M}(E)$ denote
the set of Borel probability measures supported on $E$.
For $\mu \in \mathcal{M}(E)$, the \emph{logarithmic energy}  of $\mu$ is defined by
\[
I(\mu) = \iint \log \frac{1}{|z - w|} \, d\mu(z)\, d\mu(w).
\]

The \emph{equilibrium measure} $\rho_E$ of $E$ is the unique measure in $\mathcal{M}(E)$
that minimizes the energy:
\[
I(\rho_E) = \inf_{\mu \in \mathcal{M}(E)} I(\mu)
\]
and this quantity is called the Robin constant for $E$.

The \emph{logarithmic capacity} (or simply \emph{capacity}) of $E$ is defined by
\[
\operatorname{cap}(E) = \exp\bigl(- I(\rho_E)\bigr),
\]
where $\mu_E$ is the equilibrium measure of $E$.
\end{definition}

\textbf{Example:} For $E = [-1,1]$, the equilibrium measure is given by
\begin{align*}
    d\rho_E(x) = \frac{1}{\pi} \frac{dx}{\sqrt{1 - x^2}}, \qquad x \in (-1,1).
\end{align*}

\begin{definition}
Let $\mu$ be a finite positive Borel measure on $\mathbb{R}$ with compact support,
and let $\{\tilde{P}_n\}_{n=0}^\infty$ be the sequence of orthonormal polynomials with respect to $\mu$,
with leading coefficients $\kappa_n > 0$, i.e.
\[
\tilde{P}_n(x) = \kappa_n x^n + \cdots.
\]
The measure $\mu$ is called \emph{regular} (in the sense of Stahl--Totik) if
\begin{align}\label{eq:leading_coefficient}
\lim_{n \to \infty} \kappa_n^{1/n}
= \frac{1}{\operatorname{cap}(\operatorname{supp}\mu)}.
\end{align}
\end{definition}

We follow with a sufficient condition \cite[Theorem 1.12]{Simon2007_equilibrium}, \cite{Widom1967}
on the measure $\mu$  for $\mu$ to be regular.
\begin{theorem}
\label{thm:regular_measure_sufficient}
    Let $\mu$ be a measure on $\R$ with compact support and $$E = \sigma_{ess}(d\mu) = \supp \mu_{ac}, \quad \text{cap}(E)>0.$$ 
    Suppose $d \rho_E$ is the equilibrium measure for $E$ and
    \begin{align*}
        d\mu = w(x) d\rho_e(x) + d\mu_s,
    \end{align*}
    where $d\mu_s$ is $d\rho_E$ singular.
    Suppose $w(x)>0$ for  a.e. x with respect to $\rho_E$.
    Then $\mu$ is regular.
\end{theorem}

One of the main ingredients for our algorithm is the following result \cite[Theorem 2.1]{Totik2009} that guarantees asymptotically linear spacing of zeros of orthogonal polynomials for regular measures.
\begin{theorem} \label{thm:linear_spacing}
Let $\mu$ be a regular measure with compact support $E \subset \mathbb{R}$.
Let $S \subset \operatorname{Int}(E)$ be a compact subset of the interior of $E$,
and assume that $\mu$ is absolutely continuous in a neighborhood of $S$ with
density $w(x)$ that is continuous and positive on $S$.
Let $\rho_E$ denote the equilibrium density.

Let 
$\{x_{k,n}\}_{k=1}^n$ denote the zeros of the $n$-th orthonormal polynomial
associated with $\mu$, ordered increasingly.

Then,
\begin{align*}
    \lim_{n \to \infty} n\bigl(x_{k+1,n} - x_{k,n}\bigr)\,d\rho_E(x) = 1,
\end{align*}
 and $\exists c>0, \forall x \in S, \exists k \in\{1,2,...,n\}:$
\begin{align} \label{eq:spacing}
    |x_{k,n} - x| \le \frac{c}{n}.
\end{align}
\end{theorem}

\begin{corollary}
    The measures $\mu$ satisfying Assumption \ref{as:measure_base} we consider are regular due to Theorem \ref{thm:regular_measure_sufficient}.
    Therefore, the spacing of zeros is asymptotically linear inside $\supp \mu_{ac}$ in the sense of Theorem \ref{thm:linear_spacing}.
\end{corollary}

The proofs of Theorem \ref{thm:distinguish}, \ref{thm:distinguish_outside} 
and \ref{thm:distinguish_general}
are very similar.
In all cases, we prove the statements by splitting roots into several groups based on whether they are isolated or have neighbors and discussing the implications.

\begin{proof}[Proof of Theorem \ref{thm:distinguish}]

Let $\mu $ satisfy Assumption \ref{item:1interval}.
This measure is regular.
Therefore, we can
     employ Theorem \ref{thm:linear_spacing} by 
    putting $S= [a + \delta/2, b -\delta/2]$ and $c/n < \delta/2$,
    where $c>0$ is the $\mu$ dependent constant from Theorem \ref{thm:linear_spacing}.
    As a result of Theorem \ref{lemma:atom_convergence}, which guarantees there is no pollution within the roots of $P_n$, we have three types of roots:
    \begin{itemize}   
        \item Roots inside $S$.
        \item Roots in $\supp \mu_{ac} \setminus S = [a, a +\delta/2] \cup [b-\delta/2,b]$.
        \item Roots outside $\supp \mu_{ac}$ associated with atoms and there is precisely $r$ of them.
     \end{itemize}
    If we choose $x= x_{ k \pm 1,n}$ in (\ref{eq:spacing}), we are guaranteed that two neighboring roots inside $S$ are closer than $c/n$.
    Demanding that $c/n <  \delta/2$, we obtain that all roots of $P_n$ inside $S$ are at least $\delta /2 $ close.
    
    Choosing $x=a+\delta /2$ in  (\ref{eq:spacing}), we obtain that the root $x_{l,n}$ closest to the boundary $a$ of the interval is at most $\delta= \delta/2+\delta/2$ far from the boundary.
    Similarly for the boundary $b$.
    Then, we are guaranteed that all the roots in $\supp \mu_{ac} \setminus S$ have a $\delta$ close neighbor that is $x_{l,n}$.
    We can conclude that every root of $P_n$ that lies in $\supp \mu_{ac}$ has at least one $\delta$ close neighbor in the set of roots.
    
    Moreover, due to Theorem \ref{lemma:atom_convergence} and the exponential convergence, we can assume that the $r$ roots of $P_n$  associated with the atoms are $\delta/2$ close to their corresponding atoms.
    Since we have $2\delta < \Delta$, any root outside $\supp \mu_{ac}$ cannot contain any other root in its $\delta$ neighborhood.
    This completes the proof if $\mu$ satisfies Assumption \ref{as:single_interval}.
\end{proof}

\begin{proof}
[Proof of Theorem \ref{thm:distinguish_outside}]
    If we replace Assumption \ref{as:single_interval} by Assumption \ref{as:atoms_outside},
    we have three types of roots:
    \begin{itemize}
        \item Roots outside $\mathrm{conv hull } \supp \mu_{ac}$ (we show they still correspond to atoms).
        \item Roots inside $\supp \mu_{ac}$. 
        \item The pollution, i.e. roots in $\mathrm{conv hull } \supp \mu_{ac} \setminus \supp \mu_{ac}$.
        
    \end{itemize}
    
    In this case, we lose Theorem \ref{lemma:atom_convergence}.    
    However, to treat the roots outside the continuous part $\mathrm{conv hull } \supp \mu_{ac}$,  we can use points \ref{thm:measure_0_sets} and \ref{thm:density_of_zeros} of Theorem \ref{thm:properties_og_poly} to guarantee that for $n$ large enough (without a provable upper bound), the polynomials $P_n$ will contain precisely $r$ roots outside of $\mathrm{conv hull } \supp \mu_{ac}$ and they will be at most $\delta/2$ far away from their corresponding atoms, and therefore at least $\delta$ separated.

    The bulk behavior of zeros within the $\supp \mu_{ac}$ remains the same by Theorem \ref{thm:linear_spacing} 
    as in the proof above: if we choose $   c/n < \delta/2$, then the distance of neighboring zeros inside $\supp \mu_{ac}$ is at most $\delta$. 
    
    Therefore, the only thing that changes is the fact that point \ref{thm:measure_0_sets} of Theorem \ref{thm:properties_og_poly} now allows (at most) one root in  between the intervals, that does not correspond to any component of $\supp \mu$.
    Let $y$ be the polluting root of $P_n$
    in between the intervals $[a,b]$, $[c,d]$ supporting $\mu_{ac}$.
    We have that $|c-b|\geq2\delta$.
    We distinguish two cases:
    \begin{itemize}
        \item $\dist (y, \{b,c\}) \leq \delta$ and in this case, $y$ can have a $\delta$ close neighbor within the roots of $P_n$. So $y$ is a pollution but still approximates the endpoint of the corresponding interval $\delta$-accurately as $y \in \supp \mu_{ac} + B(0,\delta)$.
        \item 
         $\dist (y, \{b,c\}) >\delta$ and $y$ cannot have any $\delta$ close neighboring root. 
        In this case, $y$ is a pollution that we can rule out because it is surrounded by two continuous parts.
    \end{itemize}
\end{proof}

Let $\{\tilde{P}_n\}_n$ be the \emph{orthonormal} polynomials associated with $\mu$.
These polynomials obey a rescaled three term recurrence relation
\begin{align}
\label{eq:orthonormal_three_term}
    x \tilde{P}_n = a_{n+1} \tilde{P}_{n+1} + b_n \tilde{P}_{n} + a_n \tilde{P}_{n-1}.
\end{align}
It is a standard result 
\cite{DamanikPushnitskiSimon2008} , \cite{SwiderskiTrojan2024}, that the  coefficients $a_n$ are uniformly bounded if and only if the measure is compactly supported. Therefore, for measures satisfying Assumption \ref{as:measure_base},
\begin{align}
\label{eq:three_term_bound}
  \sup_n a_n < a_\infty < \infty.  
\end{align}

The following  \cite{DenisovSimon2005}[Theorem 1.1] provides bounds for the amount of pollution within the roots of $P_n$.
This is key for us to be able to distinguish between roots associated with atoms and pollution.
\begin{theorem}
\label{thm:pollution}
    Let $\mu$ be a measure satisfying Assumption \ref{as:measure_base}.
     Let $\{\tilde{P}_n\}_n$ be the  polynomials orthonormal with respect to $\mu$.
    Let  $d= \dist (x_0, \supp(\mu))>0$.
    Let $\delta_\infty=d^2/(d+\sqrt 2 a_\infty)$, where $a_\infty $ is the uniform bound from (\ref{eq:three_term_bound}) for the recurrence coefficient $a_n$ in (\ref{eq:orthonormal_three_term}).
    Then either $\tilde{P}_n$ or $\tilde{P}_{n+1}$ has no zeros in $(x_0- \delta_\infty, x_0+\delta_\infty)$.
\end{theorem}

\begin{proof}[Proof of Theorem \ref{thm:distinguish_general}]
In the situation, when we only have Assumption \ref{as:measure_base}, we can still use the same approach as in the proofs above but we need to deal with the problem of distinguishing polluting roots from roots corresponding to atoms
in between the intervals supporting $\mu_{ac}$.
Given a root $y$  of $P_n$ in between the intervals supporting $\mu_{ac}$, we provide a criterion in form of inspecting neighbors of $y$ in roots of $P_{n+1}$.
For simplicity, let us consider measure with
\begin{align*}
    \supp \mu = [a_1,b_1] \cup \{x_1\} \cup [a_2,b_2], \ \text{ and }\ a_1<b_1<x_1<a_2<b_2 \text{ are  $>2\delta$ separated}.
\end{align*}
This time, we need  roots of both $P_n$ and $P_{n+1}$ for $n$ sufficiently large (we explain the choice of $n$ below).
First, we split the roots into two groups:
\begin{itemize}
    \item Roots in $\supp \mu_{ac} + B(0,\delta).$
    \item Roots in $\mathrm{convhull } \ \supp \mu_{ac} \setminus (\supp \mu_{ac} + B(0,\delta)) = [b_1+\delta,a_2-\delta]$.
\end{itemize}

For the roots in $\supp \mu_{ac} + B(0,\delta)$, we can assume $n$ large enough so that these roots appear in $\delta$ close bulks of at least  three due to Theorem \ref{thm:linear_spacing}.
Also, we can assume $n$ large enough so that if some root falls inside $(b_1,b_1+\delta)$ or $(a_2-\delta,a_2)$, this root has a $\delta$ close neighbor in one of the bulks  associated with the continuous part.

Next, we study roots in $[b_1+\delta,a_2-\delta]$.
We set $\rho = \delta^2/(\delta + \sqrt2 a_\infty)$.
Note that $\rho < \delta.$
We use the following:
\begin{enumerate} 
    \item  \label{item:1} Point \ref{thm:density_of_zeros} of Theorem \ref{thm:properties_og_poly} states that for $n$ sufficiently large, $P_n$ will have at least one root in $(x_1-\rho,x_1+\rho)$ and therefore also in $(x_1-\delta,x_1+\delta)$.
    \item  \label{item:2} Point \ref{thm:measure_0_sets} of Theorem \ref{thm:pollution} states that we might have one or no root in $(b_1+\delta, x_1)$ and one or no root in $(x_1,a_2-\delta)$.
    \item Lastly, Theorem \ref{thm:pollution} states that for a point $y \in(b_1+\delta,x_1-\delta)\cup (x_1+\delta,a_2-\delta)$,
either $P_n$ or $P_{n+1}$ has no roots in $(y-\rho,y+\rho)$.  \label{item:3}
\end{enumerate}

With these observations, we classify the following situations that might occur:
\begin{itemize}
    \item For both $P_n$ and $P_{n+1}$, there is exactly one root $x_{i,n}, x_{l,n+1}$, respectively in $(b_1+\delta,a_2-\delta)$, it is precisely the one root that is associated with the atom $x_1$ and both $x_{i,n}, x_{l,n+1} \in (x_1-\rho,x_1+\rho)$ due to \ref{item:1}.
    \item 
    $P_n$ has two roots $x_{i,n}, x_{i+1,n}$ in $(b_1+\delta,a_2-\delta)$ and
     these two roots are $\delta$ close. 
     \ref{item:2} states that there must be a point of $\supp \mu$ in between them.
     Then, surely, $(x_{i,n}, x_{i+1,n})/2$ is $\delta$ close to the atom $x_1$.
     \item 
     $P_n$ has two roots $x_{i,n}, x_{i+1,n}$ in $(b_1+\delta,a_2-\delta)$ and
     these two roots are $\delta$ separated.
\end{itemize}
We elaborate on the situation when $P_n$ has two $\delta$ separated roots $x_{i,n}, x_{i+1,n}$ and we explain how to determine which one is associated with the atom $x_1$ and which one is a pollution.
Now, we need to also inspect the roots of $P_{n+1}$.
We further distinguish two situations.
If $P_{n+1}$ has only one root $(b_1+\delta,a_2-\delta)$, we know this root $x_{l,n+1}$ must be associated with the atom $x_1$.
Without loss of generality, let $x_{i,n}$ be the root closer to $x_{l,n+1}$. Then, due to \ref{item:1}, 
\begin{align*}
    |x_{i,n} - x_{l,n+1}| \leq
    |x_1 - x_{i,n}| + |x_1-x_{l,n+1}|<\rho/2+\rho/2= \rho,
    \\
    |x_{i+1,n} - x_{l,n+1}|>
    | |x_{i+1,n} - x_{i,n}| - |x_{i,n} - x_{l,n+1}| | > |\delta - \rho | > \rho.
\end{align*}
 
Lastly, we assume $P_{n+1}$ also
has two roots $x_{l,n+1}, x_{l+1,n+1}$ in $(b_1+\delta,a_2-\delta)$ that are $\delta$ separated.
Let us consider two sets $A = (x_{i,n}- \rho ,x_{i,n} +\rho) $, $B = (x_{i+1,n}- \rho ,x_{i+1,n} +\rho)$.
Without loss of generality, let $x_{i,n}$, $x_{l,n+1}$ be the roots associated with the atom $x_1$ and let $x_{i+1,n},x_{l+1,n+1}$ be the pollution.
Due to \ref{item:1},
we observe that for the root of $P_n$ associated with an atom, we can find a $\rho$ close neighbor in roots of $P_n$ 
\begin{align*}
    | x_{i,n} - x_{l,n+1} | < | x_{i,n} - x_1| + |x_1 - x_{l,n+1} | < \rho/2 + \rho/2 = \rho. 
\end{align*}
Moreover, for the polluting root of $P_n$, we can bound the distance from all roots of $P_{n+1}$ by $\rho$
\begin{align*}
    &|x_{i+1,n} - x_{l,n+1}| > |x_{i+1,n} - x_{i,n}| - |x_{i,n} - x_{l,n+1}|   | > \delta - \rho > \rho,
    \\
   & |x_{i+1,n} - x_{l+1,n+1}| > \rho \text{ due to } \ref{item:3},
   \\
   &|x_{i+1,n} - x_{j,n+1}| > \rho, \ j\not = l,l+1 \text{ due to } x_{j,n+1} \in \supp \mu_{ac} + B(0,\delta).
\end{align*}

\end{proof}

\subsection{Proof of Theorem \ref{thm:inclusions_and_isomorphisms}}
\label{proof:inclusions_and_isomorphisms}
\begin{proof}
    First, let us note that since we are in dimension one, for all $n \in \mathbb{N}$, one of the following must happen:
    either $\rank M_n = \rank M_{n+1}$, or
         $\ker \phi_n = \ker \phi_{n+1}$.
    Moreover, if the flat condition is satisfied for some $n_0$, it is satisfied for all $n\geq n_0$.
    
    Now, we show that if the flat extension condition is not satisfied for $n \in \mathbb{N}$, then the following set-inclusion holds
    \begin{align} \label{eq:non_flat_inclusion}
        \mathcal{K}_n \subsetneq \mathcal{K}_{n+1}.
    \end{align}
    Since $\ker \phi_n = \ker \phi_{n+1}$, we have that
    \begin{align*}
        [p]_{\phi_n}
        =
        [p]_{\phi_{n+1}}, \quad \text{for } p \in \R[x]_n.
    \end{align*}
    Therefore,
    \begin{align*}
        \left([p]_{\phi_n} \in \mathcal{K}_n \right)
        \Longrightarrow
        \left([p]_{\phi_n} \in \mathcal{K}_{n+1} \right), \quad \text{for } p \in \R[x]_n.
    \end{align*}
    At the same time, $x^{n+1} \not \in \ker \phi_{n+1}$.
    This shows (\ref{eq:non_flat_inclusion}).

    If now the flat extension condition holds, we refer the reader to \cite[Proposition 5.16]{LopezQuijorna2021DetectingGNS}, which states that
    the mapping
    \begin{align}
    \label{eq:canonical_isomorphism}
        \iota_n : \mathcal{K}_n \to 
        \mathcal{K}_{n+1}: [p]_{\phi_n} \mapsto [p]_{\phi_{n+1}},
    \end{align}
     is an isomorphism between $\mathcal{K}_n$ and $\mathcal{K}_{n+1}$,
    meaning 
    \begin{align*}
        \mathcal{K}_n \simeq \mathcal{K}_{n+1}.
    \end{align*}
\end{proof}

\subsection{Proof of Lemma \ref{lemma:isomorphism}}
\label{proof:isomorphisms}
\begin{proof} 

\begin{enumerate}
    \item We factorize the same set in (\ref{eq:set_equality}), so it suffices to show that the two equivalences by which we factorize are the same. 
    For $p_1,p_2 \in \R[x]_n$, we have that
    \begin{align*}
        p_1 \sim p_2 \text{ w.r.t. } \ker \phi_n
        &\Longleftrightarrow 
        p_1 - p_2 \in \ker \phi_n
        \\
        &\Longleftrightarrow 
        \phi_n(p_1-p_2,p_1-p_2) =0
         \\
        &\Longleftrightarrow
        \int (p_1(x)-p_2(x))^2 d\mu(x)
         \\
        &\Longleftrightarrow
        p_1-p_2 = 0 \text{ a.e. w.r.t. } \mu
         \\
        &\Longleftrightarrow
        p_1 \sim p_2 \text{ in the sense } \mu  \text{ equal a.e.}
        \end{align*}
    \item Let us write a precise definition of the equivalence classes spanning both of the sets in (\ref{eq:Hilbert_space_equality}). 
    Take $p \in \R[x]_n, $ then
    \begin{align*}
      \R[x]_n \Big|_{\mu}  \ni [p]_{\mu;\R[x]_n} = \{ f:\R \to \R \text{ polynomials}, \ \deg f \leq n; \ f = p \ \mu-\text{a.e.} \},
       \\
      \mathcal{H}_n\ni [p]_{\mu;L^2(\R,\mu)} = \{ f:\R \to \R \text{ measurable functions}; \ f = p \ \mu-\text{a.e.} \}.
    \end{align*}
    It is obvious that these two sets are different, to be specific, $[p]_{\mu;\R[x]_n}$ is a strict subset of $[p]_{\mu;L^2(\R,\mu)}$. 
    Let us define
\begin{align*}
    T:
\mathbb{R}[x]_n\big|_\mu
\longrightarrow
\mathcal{H}_n:[p]_{\mu;\mathbb{R}[x]_n} \mapsto[p]_{\mu;L^2(\R,\mu)}.
\end{align*}
We show that $T$ is an isometric isomorphism. It is obviously a linear operator acting between two finite-dimensional Hilbert spaces 
\begin{enumerate}
    \item $T$ is well defined and injective. Take any $p_1,p_2 \in \R[x]_n$, then
    \begin{align*}
        &[p_1]_{\mu;\R[x]_n} = [p_2]_{\mu;\R[x]_n}
        \Longleftrightarrow p_1= p_2 \text{ a.e. w.r.t. } \mu
        \Longleftrightarrow \\\Longleftrightarrow 
        &T[p_1]_{\mu;\R[x]_n} = [p_1]_{\mu;L^2(\R,\mu)} = [p_2]_{\mu;L^2(\R,\mu)} = T[p_2]_{\mu;\R[x]_n}
    \end{align*}
    \item $T$ is onto $\mathcal{H}_n$. Let us take an arbitrary element $[h]_{\mu;L^2(\R,\mu)} \in \mathcal{H}_n$, where
    \begin{align*}
        h(x) = \sum_{k=0}^n c_kx^k.
    \end{align*}
    Then, if we define
    \begin{align*}
        [p(x)]_{\mu;\R[x]_n} = \left[ \sum_{k=0}^n c_kx^k\right]_{\mu;\R[x]_n},
    \end{align*}
    we get the desired element in $[h]_{\mu;L^2(\R,\mu)}\in \mathcal
    {H}_n$ as an image of $[p]_{\mu;\R[x]_n}$: \begin{align*}
        T[p(x)]_{\mu;\R[x]_n} = [h]_{\mu;L^2(\R,\mu)}.
    \end{align*} 
    \item $T$ is an isometry:
    \begin{align*}
        \begin{Vmatrix}
            [p]_{\mu;\R[x]_n}
        \end{Vmatrix}^2_{(\R[x]_n; \phi_n)} = \phi_n(p,p)=\int p^2(x)d\mu(x) =\begin{Vmatrix}
            T[p]_{\mu;\R[x]_n}
        \end{Vmatrix}^2_{L^2(\R,\mu)}.
    \end{align*}
\end{enumerate}
\end{enumerate}    
\end{proof}

\section{Additional Experimental Results}

The following Figure \ref{fig:two_polynomials} demonstrates the behavior of two consecutive orthogonal polynomials generated by a moment matrix.
This figure shows:
\begin{itemize}
    \item Interlacing property of orthogonal polynomials \ref{thm:interlacing}.
    \item Convergence of the corresponding roots to the atom \ref{lemma:atom_convergence}.
    \item All of the roots in Figure \ref{fig:two_polynomials} are supported in the convex hull of the support of the underlying measure as we would expect due to Theorem \ref{thm:roots_and_support}.

    \item 
    Most of the roots are even concentrated inside $\supp \mu$.
    However, we can observe two roots polluting the interval $[c-r/3, c+r/3]$, slightly worsening the IOU metric compared to the situation if they were absent.
    
    \item 
    These polluting roots are even expected by the theory, since 
    measure zero sets are allowed to have one root due to point \ref{thm:measure_0_sets} of Theorem   
\ref{thm:properties_og_poly}.
\end{itemize}
\begin{figure} [H]
    \centering
    \includegraphics[width=0.7\linewidth]{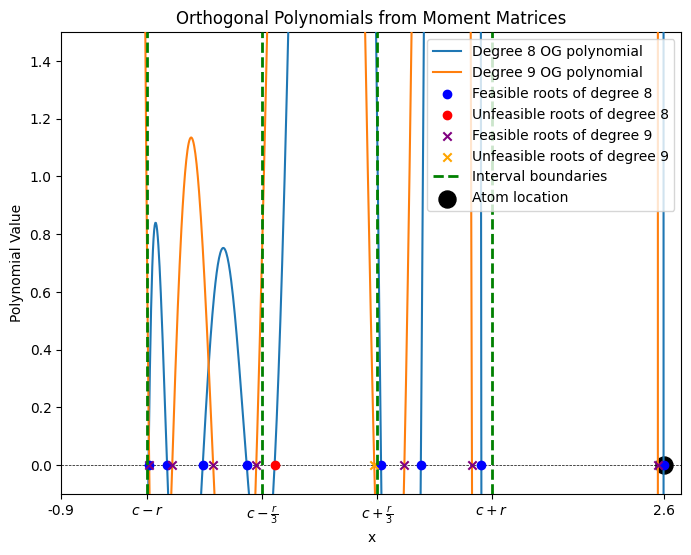}
    \caption{Behavior of two consecutive orthogonal polynomials $p_8(x), p_{9}(x)$ for $(a,c,r) = (1.0,0.6,1.0)$.}
    \label{fig:two_polynomials}
\end{figure}

\subsection{TSSOS Robust extraction vs Orthogonal polynomials}
\label{Sec:TSSOS_vs_OGPoly}
In the Section \ref{sec:num} above, we study a sequence of orthogonal polynomials generated by a sequence of moment matrices corresponding to a measure $\mu$ satisfying Assumption \ref{as:measure_base}.
Namely, we analyze the behavior of the zeros of these orthogonal polynomials.
This section focuses on how we obtain the aforementioned zeros of the orthogonal polynomials.

To be specific, we
compare our algorithm $\texttt{roots}$ \ref{alg:roots_from_moments} with the 
TSSOS routine \verb|extract_solutions_robust| \citep{wang2021tssos}.
The function \verb|extract_solutions_robust| takes as input the solution of an SDP arising from a moment relaxation of a polynomial optimization problem (POP) and attempts to extract a minimizer of the POP from the associated moment matrix.

We show that when the \textbf{rtol} parameter of \verb|extract_solutions_robust| is set to zero, its output coincides with that of Algorithm~\ref{alg:roots_from_moments}. This observation provides numerical evidence that the interpretation of the \verb|extract_solutions_robust| in Remark \ref{remark:TSSOS} is correct.
We consider same moment matrices $M_N$ as in the previous Section \ref{sec:num}.

The reported statistics are based on 120 instances of the POP defined in~(\ref{eq:experiments_2}), generated by sampling the parameters $a$, $c$, and $r$. For each instance, we computed the moment relaxations at various relaxation orders and compared the outputs of the two algorithms for the corresponding pseudomoment matrices.

\begin{table}[th]
\centering
\caption{Numerical comparison between TSSOS and the proposed method over 120 polynomial optimization instances.}
\label{tab:tssos_comparison}
\begin{tabular}{lccc}
\hline
Metric & Value \\
\hline
Number of test instances & 120 \\
Cumulative absolute difference & $3.41 \times 10^{-11}$ \\
Maximum absolute difference (single instance) & $1.17 \times 10^{-11}$ \\
Minimum absolute difference (single instance) & $5.12 \times 10^{-17}$ \\
Number of non-PSD moment matrices & 2 \\
\hline
\end{tabular}
\end{table}

\begin{figure} [h]
    \centering
    \includegraphics[width=0.62\linewidth]{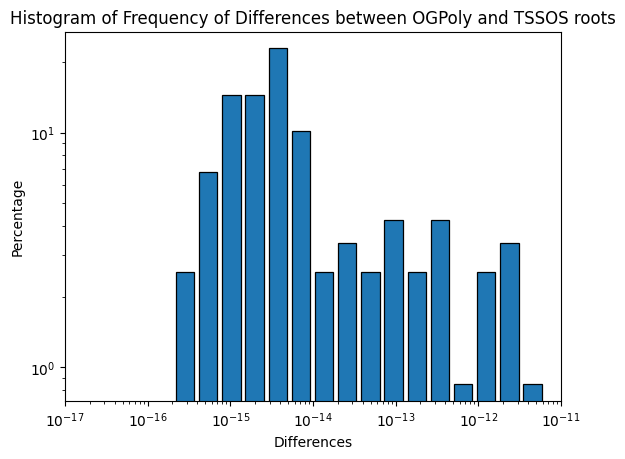}
    \caption{Histogram of frequency of absolute differences between the outputs of TSSOS and OGPoly, displayed on a logarithmic scale. }
    \label{fig:placeholder}
\end{figure}

We conclude that the outputs of Algorithm~\ref{alg:roots_from_moments} and \verb|extract_solutions_robust| agree up to numerical errors at the level of machine precision.

\end{document}

%% file: math_commands.tex
\usepackage{amsmath,amsfonts,bm}

\def\eqref#1{equation~\ref{#1}}

\def\1{\bm{1}}

\DeclareMathAlphabet{\mathsfit}{\encodingdefault}{\sfdefault}{m}{sl}
\SetMathAlphabet{\mathsfit}{bold}{\encodingdefault}{\sfdefault}{bx}{n}

\newcommand{\R}{\mathbb{R}}

